\documentclass{amsart}

\usepackage{amssymb,amsmath,graphicx}

\usepackage{color}

\newtoks\prt

\numberwithin{equation}{section}

\newtheorem{thm}{Theorem}[section]

\newtheorem{lemma}[thm]{Lemma}

\newtheorem{prop}[thm]{Proposition}

\theoremstyle{definition}

\newtheorem{definition}[thm]{Definition}

\newtheorem{remarks}[thm]{Remarks}

\def\eqn#1$$#2$${\begin{equation}\label#1#2\end{equation}}

\def\fra{\mathfrak{A}}

\def\A{\mathcal A}

\def\B{\mathcal B}

\def\C{\mathcal C}

\def\E{\mathcal E}

\def\e{e^\ast}

\def\s{u^\ast}

\def\t{v^\ast}

\def\E{E^\ast}

\def\H{\mathcal H}

\def\M{\mathcal M}

\def\ce{\mathbb C}

\def\lin{Lindel\"of}

\def\ep{\varepsilon}

\def\en{\mathbb N}

\def\ef{\mathbb F}

\def\er{\mathbb R}

\def\ov{\overline}

\def \Ch {\operatorname{Ch}}

\def \ext {\operatorname{ext}}

\def\wh{\widehat}

\def \reg {\partial _{\kern1pt\text{reg}}}

\def\la{\langle}

\def\ra{\rangle}

\def\re{\Re}

\def\di{\,\mbox{\rm d}}

\newcommand{\norm}[1]{\left\|#1\right\|}

\newcommand{\abs}[1]{\left| #1  \right|}

\renewcommand{\Re}{\operatorname{Re}}

\begin{document}

\title[Isomorphisms of subspaces of vector-valued continuous functions]{Isomorphisms of subspaces of vector-valued continuous functions}

\author{Jakub Rondo\v s and Ji\v r\'\i\ Spurn\'y}

\address{Charles University\\
Faculty of Mathematics and Physics\\
Department of Mathematical Analysis \\
Sokolovsk\'{a} 83, 186 \ 75\\Praha 8, Czech Republic}

\email{spurny@karlin.mff.cuni.cz}
\email{jakub.rondos@gmail.com}



\thanks{The research was supported by the Research grant  GA\v{C}R 17-00941S}

\begin{abstract}
We deal with isomorphic Banach-Stone type theorems for closed subspaces of vector-valued continuous functions. Let $\ef=\er$ or $\ce$. For $i=1,2$, let $E_i$ be a reflexive Banach space over $\ef$ with a certain parameter $\lambda(E_i)>1$, which in the real case coincides with the Schaffer constant of $E_i$, let $K_i$ be a locally compact (Hausdorff) topological space and let $\H_i$ be a closed subspace of $\C_0(K_i, E_i)$ such that each point of the Choquet boundary $\Ch_{\H_i} K_i$ of $\H_i$ is a weak peak point. We show that if there exists an isomorphism $T\colon \H_1\to \H_2$ with $\norm{T}\cdot \norm{T^{-1}}<\min \lbrace \lambda(E_1), \lambda(E_2) \rbrace$, then $\Ch_{\H_1} K_1$ is homeomorphic to $\Ch_{\H_2} K_2$. Next we provide an analogous version of the weak vector-valued Banach-Stone theorem for subspaces, where the target spaces do not contain an isomorphic copy of $c_0$.
\end{abstract}

\maketitle


\section{Introduction}
We work within the framework of real or complex vector spaces and write $\ef$ for the respective field $\er$ or $\ce$. 
If $E$ is a Banach space then $E^*$ stands for its dual space. We denote by $B_E$ and $S_E$ the unit ball and sphere in $E$, respectively, and we write $\la\cdot,\cdot\ra\colon E^*\times E\to\ef$ for the duality mapping. For a locally compact (Hausdorff) space $K$, let $\C_0(K,E)$ denote the space of all continuous $E$-valued functions vanishing at infinity. We consider this space endowed with the sup-norm
\[\norm{f}_{\sup}=\sup_{x \in K} \norm{f(x)}, \quad f \in \C_0(K,E).\]
 If $K$ is compact, then this space will be denoted by $\C(K,E)$. For a compact space $K$, we identify the dual space $(\C(K,E))^*$ with the space $\M(K,\E)$ of all $\E$-valued Radon measures on $K$ endowed with the variation norm via Singer's theorem (see \cite[p.192]{Singer}). Thus $\M(K, \ef)$ is the usual set of Radon measures on $K$. Unless otherwise stated, we consider $\M(K,\E)$ endowed with the weak$^*$ topology given by this duality.

Our starting point is the classical Banach-Stone theorem which asserts that, given a pair of compact spaces $K$ and $L$, they are homeomorphic provided $\C(K,\ef)$ is isometric to $\C(L,\ef)$ (see \cite[Theorem 3.117]{faspol}).

The first direction of our research are the so called isomorphic Banach-Stone type theorems, where the assumption of the isometry between the spaces of continuous functions is replaced by an isomorphism $T:\C(K,\ef) \rightarrow \C(L,\ef)$ with $\norm{T}\cdot \norm{T^{-1}}$ being small.

A remarkable generalization of the Banach-Stone theorem in this way was given independently by Amir \cite{amir} and Cambern \cite{cambern}. They showed that compact spaces $K$ and $L$ are homeomorphic if there exists an isomorphism $T\colon \C(K,\ef)\to\C(L,\ef)$ with $\norm{T}\cdot \norm{T^{-1}}<2$. Alternative proofs were given by Cohen \cite{cohen} and Drewnowski \cite{drewnow}.

The latest result in the direction of the Amir-Cambern theorem is due to E.M. Galego and A.L.~Porto da Silva in \cite{galego-silva} who proved the following theorem. If $T$ is a function from $\C_0(K,\er)$ to $\C_0(S,\er)$, $T(0)=0$, and both $T$ and $T^{-1}$ are bijective coarse $(M,1)$-quasi-isometries with $M<\sqrt{2}$, then  $K$ and $S$ are homeomorphic and  there exists a homeomorphism $\phi$ from $S$ to $K$ and a continuous function $\lambda\colon S\to \{-1,1\}$ such that for all $s\in S$ and $f\in \C_0(K,\er)$ one has
\[
\norm{MTf(s)-\lambda(s)f(\phi(s))}\leq (M^2-1)\norm{f}+\delta,
\]
where $\delta$  does not depend on $f$ or $s$.

A generalization of the theorem of Amir and Cambern to the context of affine continuous functions on compact convex sets was given by Chu and Cohen in \cite{cc}. In order to explain their results we need a bit of terminology.
By a compact convex set we mean a compact convex subset of a locally convex (Hausdorff) space. Let $\fra(X,\ef)$ be the space of all continuous $\ef$-valued affine functions on a compact convex set $X$ endowed with the sup-norm. 
Let $\M^+(X)$ and $\M^1(X)$ stand for the sets of nonnegative and probability Radon measures on $X$, respectively. For any $\mu\in\M^1(X)$ there exists a unique point $r(\mu)\in X$ such that $\mu(a)=a(r(\mu))$, $a\in \fra(X,\ef)$, see \cite[Proposition~I.2.1]{alfsen}.
We call $r(\mu)$ the \emph{barycenter} of $\mu$, or alternatively, we say that $\mu$ \emph{represents} the point $r(\mu)$.
If $\mu,\nu\in\M^+(X)$, then $\mu\prec \nu$ if $\mu(k)\le \nu(k)$ for each convex continuous function $k$ on $X$. A measure $\mu\in\M^+(X)$ is \emph{maximal} if $\mu$ is $\prec$-maximal.

By the Choquet--Bishop--de-Leeuw representation theorem (see \cite[Theorem~I.4.8]{alfsen}), for each $x\in X$ there exists a maximal measure $\mu\in\M^1(X)$ with $r(\mu)=x$. If this measure is uniquely determined for each $x\in X$, the set $X$ is called a \emph{simplex}. It is called a \emph{Bauer simplex} if, moreover, the set $\ext X$ of extreme points of $X$ is closed. In this case, the space $\fra(X,\ef)$ is isometric to the space $\C(\ext X,\ef)$ (see \cite[Theorem~II.4.3]{alfsen}).
On the other hand, given a space $\C(K,\ef)$, it is isometric to the space $\fra(\M^1(K),\ef)$, see (\cite[Corollary~II.4.2]{alfsen}).

A reformulation of the result of Amir and Cambern for simplices reads as follows: Given Bauer simplices $X$ and $Y$, the sets $\ext X$ and $\ext Y$ are homeomorphic, provided there exists an isomorphism $T\colon \fra(X,\ef)\to\fra(Y,\ef)$ with $\norm{T}\cdot \norm{T^{-1}}<2$.

The aforementioned  Chu and Cohen proved in \cite{cc} that for  compact convex sets $X$ and $Y$, the sets $\ext X$ and $\ext Y$ are homeomorphic provided there exists an isomorphism $T\colon \fra(X,\er)\to\fra(Y,\er)$ with $\norm{T}\cdot \norm{T^{-1}}<2$ and
one of the following conditions hold:

(i) $X$ and $Y$ are  simplices such that their extreme points are weak peak points;

(ii) $X$ and $Y$ are metrizable and their extreme points are weak peak points;

(iii) $\ext X$ and $\ext Y$ are closed and extreme points of $X$ and $Y$ are split faces.

A point $x\in X$ is a \emph{weak peak point} if given $\ep\in (0,1)$ and an open set $U\subset X$ containing $x$, there exists $a$ in $B_{\fra(X,\er)}$ such that $\abs{a}<\ep$ on $\ext X\setminus U$ and $a(x)>1-\ep$, see \cite[p. 73]{cc}.

In \cite{lusppams}, it was showed that extreme points of $X$ and $Y$ are homeomorphic, provided there exists an isomorphism $T\colon \fra(X,\er)\to\fra(Y,\er)$ with $\norm{T}\cdot \norm{T^{-1}}<2$, extreme points are weak peak points and both $\ext X$ and $\ext Y$ are  \lin\ sets.

In \cite{dosp} the same result is proved without the assumption of the \lin\ property and paper \cite{rondos-spurny} provides an analogous result for the case of complex functions.
It turns out that this result is in a sense optimal since the bound $2$ cannot be improved (see \cite{cohen-bound2}, where a pair of nonhomeomorphic compact spaces $K_1,K_2$ for which there exists an isomorphism $T\colon \C(K_1,\er)\to \C(K_2, \er)$ with $\norm{T}\cdot\norm{T^{-1}}=2$ is constructed) and the assumption on weak peak points cannot be omitted (see \cite{hess}, where the author constructs for each $\ep\in (0,1)$ a pair of simplices $X_1,X_2$ such that $\ext X_1$ is not homeomorphic to $\ext X_2$ but there is an isomorphism $T\colon \fra(X_1,\er)\to \fra(X_2,\er)$ with $\norm{T}\cdot \norm{T^{-1}}<1+\ep$).

In \cite{rondos-spurny-spaces}, we have generalized the previous result by showing that for closed subspaces $\H_i \subset \C_0(K_i,\ef)$ for $i=1,2$, their Choquet boundaries are homeomorphic provided points in the Choquet boundaries are weak peak points and there exists an isomorphism $T\colon \H_1\to \H_2$ with $\norm{T}\cdot \norm{T^{-1}}<2$.  We recall that $x\in K_i$ is a \emph{weak peak point} (with respect to $\H_i$) if for a given $\ep\in (0,1)$ and a neighborhood $U$ of $x$ there exists a function $h\in B_{\H_i}$ such that $h(x)>1-\ep$ and $\abs{h}<\ep$ on $\Ch_{\H_i} K_i\setminus U$.

The first vector-valued version of the isomorphic Banach-Stone theorem is due to Cambern \cite{cambern-illinois}, who proved that if $E$ is a finite-dimensional Hilbert space and $\C_0(K_1, E)$ is isomorphic to $\C_0(K_2, E)$ by an isomorphism $T$ satisfying $\norm{T}\cdot\norm{T^{-1}}<\sqrt{2}$, then the locally compact spaces $K_1$ and $K_2$ are homeomorphic. 

Later in \cite{cambern-pacific}, Cambern proved the first result in the spirit of isomorphic vector-valued Banach-Stone theorem for infinite-dimensional Banach spaces. He showed that if $K_1$ and $K_2$ are compact spaces, $E$ is a uniformly convex Banach space and $T: \C(K_1, E) \rightarrow \C(K_2, E)$ is an isomorphism satisfying $\norm{T}\cdot\norm{T^{-1}}<(1-\delta(1))^{-1}$, then $K_1$ and $K_2$ are homeomorphic. Here $\delta:[0, 2] \rightarrow [0,1]$ denotes the modulus of convexity of $E$.  

Since then, there have been improvements in this area proved e.g. in \cite{behrends-cambern},  \cite{behrends-pacific} and \cite{jarosz-pacific}. 

Many of those results were recently unified and strengthened in \cite{cidralgalegovillamizar}, where it was showed that if $E$ is a real or complex reflexive Banach space with $\lambda(E)>1$, then for all locally compact spaces $K_1, K_2$, the existence of an isomorphism $T:\C_0(K_1, E) \rightarrow \C_0(K_2, E)$ with $\norm{T}\cdot\norm{T^{-1}}<\lambda(E)$ implies that the spaces $K_1, K_2$ are homeomorphic. Here

\[\lambda(E)=\inf \lbrace \max \lbrace \Vert e_1+\lambda e_2 \Vert: \lambda \in \ef, \abs{\lambda}=1 \rbrace: e_1, e_2 \in S_E \rbrace\]
is a parameter introduced by Jarosz in \cite{jarosz-pacific}.

It is easy to check that $\lambda(\ef)=2$, thus this result recovers the theorem of Amir and Cambern. The authors of \cite{cidralgalegovillamizar} also showed that the constant $\lambda(E)=2^{\frac{1}{p}}$ is the best possible for $E=l_p$, where $2 \leq p <\infty$. 

The properties of the parameter $\lambda$ and its relation to various other parameters of Banach spaces were described comprehensively in \cite{cidralgalegovillamizar}. Here we just mention that for a real Banach space $E$, the parameter $\lambda(E)$ is called the Schaffer constant of $E$, and the fact that $\lambda(E)>1$ implies that $E$ is reflexive (see \cite[Proposition 1 and Theorem 2]{KaMaTa} and \cite[Theorem 1.1]{james}). Also we will frequently use the fact that $1 \leq \lambda(E) \leq 2$ for each Banach space $E$.

The first main result of this paper is a generalization of the previous result in the following way.

\begin{thm}
	\label{main}
	For $i=1,2$, let $\H_i$ be a closed subspace of $\C_0(K_i, E_i)$ for some locally compact space $K_i$ and a reflexive Banach space $E_i$ over the same field $\ef$ satisfying $\lambda(E_i)>1$. Assume that each point of the Choquet boundary $\Ch_{\H_i} K_i$ of $\H_i$ is a weak peak point and let $T\colon \H_1\to \H_2$ be an isomorphism satisfying $\norm{T}\cdot\norm{T^{-1}}<\min \lbrace \lambda(E_1), \lambda(E_2) \rbrace$.
	Then $\Ch_{\H_1} K_1$ is homeomorphic to $\Ch_{\H_2} K_2$.
\end{thm}

The notions of Choquet boundary and weak peak points will be described in Section \ref{sectionaux}. For $\H=\C_0(K, E)$, the Choquet boundary of $\H$ coincides with $K$ and each point of $K$ is a weak peak point, see Remarks \ref{remarks}. The proof of the above result combines the methods of \cite{cidralgalegovillamizar} (which are in turn adapted from \cite{cambern-pacific}) with the methods developed in \cite{rondos-spurny-spaces}. The maximum principle for affine functions of the first Borel class (see Lemma \ref{minim}) even allows some technical simplifications compared to \cite{cidralgalegovillamizar} and \cite{cambern-pacific}, since we can construct the desired homeomorphism directly as a mapping from $K_1$ to $K_2$ (compare with the definition on pages 248 and 249 in \cite{cambern-pacific}). The reason for this is that with the use of the maximum principle it is much easier to prove that the desired mapping from $K_1$ to $K_2$ is surjective.

The isomorphic vector-valued Banach-Stone theorems for subspaces were treated by Al-Halees and Fleming in \cite{halees-fleming}, with quite different methods compared to ours. The authors use the notion of strong boundary, which is somewhat different from the Choquet boundary that we use. Also, their results work for those subspaces of vector-valued continuous functions that are so called $\C_0(K, \ef)$-modules, meaning that they are closed with respect to multiplication by functions from $\C_0(K, \ef)$. In problem (i) on page 213, they ask whether this module condition can be weakened or removed. We believe that our results give a positive answer to this problem.  

Next we turn our attention to the so called weak version of the Banach-Stone theorem. 

The first result in this area is due to Cengiz \cite{cengiz}, who showed that locally compact Hausdorff spaces $K_1$ and $K_2$ have the same cardinality provided that the spaces $\C_0(K_1, \ef)$ and $\C_0(K_2, \ef)$ are isomorphic.

In \cite{rondos-spurny}, we showed that if for $i=1,2$, $X_i$ is a compact convex set such that each point of $\ext X_i$ is a weak peak point, then the cardinality of $\ext X_1$ is equal to the cardinality of $\ext X_2$ provided that $\fra(X_1,\ce)$ and $\fra(X_2,\ce)$ are isomorphic. In \cite{rondos-spurny-spaces}, we provided an analogous result for the Choquet boundaries of closed subspaces of $\C_0(K_i,\ef)$.

In the area of weak vector-valued Banach-Stone type theorems, Candido and Galego in \cite{candidogalego} showed that if $K_1, K_2$ are locally compact Hausdorff spaces and $E$ is a Banach space having nontrivial Rademacher cotype, such that either $\E$ has the Radon-Nikodym property or $E$ is separable, then either both $K_1$ and $K_2$ are finite or $K_1$ and $K_2$ have the same cardinality provided that the spaces $\C_0(K_1, E)$ and $\C_0(K_2, E)$ are isomorphic. 

This result was improved by Galego and Rincón-Villamizar in \cite{GalegoVillamizar}, who showed that the same conclusion holds for Banach spaces not containing an isomorphic copy of $c_0$. The way to this improvement was using a nice characterization of Banach spaces not containing an isomorphic copy of $c_0$, see \cite[Theorem 6.7]{morrison2001functional}, and a result of Plebanek, see \cite[Theorem 3.3]{Plebanek2015}, which made it possible to remove the assumptions of separability and the Radon-Nikodym property. We prove an analogous result for closed subspaces of vector-valued continuous functions, whose Choquet boundaries consist of weak peak points.
In our setting, Plebanek's result is replaced by the maximum principle. 

Thus the second main result of this paper is the following.	
\begin{thm}
	\label{weak Banach-Stone}
	For $i=1,2$, let $\H_i$ be a closed subspace of $\C_0(K_i, E_i)$ for some locally compact space $K_i$ and a Banach space $E_i$ over the same field $\ef$. For $i=1,2$, let $E_i$ does not contain an isomorphic copy of $c_0$. Assume that each point of $\Ch_{\H_1} K_1$ and $\Ch_{\H_2} K_2$ is a weak peak point and let $T\colon \H_1\to \H_2$ be an isomorphism. Then either both sets $\Ch_{\H_1} K_1$ and $\Ch_{\H_2} K_2$ are finite or they have the same cardinality.
\end{thm}

\section{Notation and auxiliary results}
\label{sectionaux}

Let $K$ be a locally compact Hausdorff space, $E$ be a Banach space and $\H \subset \C_0(K, E)$ be a subspace.
We will from now on implicitly assume that the dimension of both the spaces $E$ and $\H$ is at least 1. If $\H$ or $E$ has the dimension zero then the assumptions of our main results are never satisfied.

For $h \in \H$ and $\e \in \E$, $\e(h)$ is the element of $\C_0(K, \ef)$ defined by $\e(h)(x)=\la \e, h(x) \ra$ for $x \in K$. We define the canonical scalar function space $\A \subset \C_0(K, \ef)$ associated to $\H$ as the closed linear span of the set 
\[ \lbrace \e(h): \e \in \E, h \in \H \rbrace  \subset \C_0(K, \ef).\]
 Since both the spaces $\H$ and $E$ are of dimension at least 1 by the assumption, it follows that the dimension of $\A$ is at least 1 as well. 
 
 The spaces $B_{\E}$, $B_{\H^*}$ and $B_{\A^*}$ will be always equipped with the $w^*$-topology, unless otherwise stated.
 We consider evaluation mappings $i$, $\phi$ defined as 
 \[ i\colon K \to B_{\A^*}, x \mapsto i(x), \quad 
 \phi\colon K \times B_{\E} \to B_{\H^*}, (x, \e) \mapsto \phi(x, \e), \]
 where 
 \[ \la i(x), a \ra=a(x), \quad a \in \A, \quad \text{and} \quad
 \la \phi(x, \e), h \ra=\la \e, h(x) \ra=\e(h)(x), \quad h \in \H.\]
 The mappings $i$ and $\phi$ are continuous. Moreover, for each $x \in K$, $\e_1, \e_2 \in B_{\E}$ and $\alpha_1, \alpha_2 \in B_{\ef}$, if the element $\alpha_1\e_1+\alpha_2\e_2$ belongs to $B_{\E}$ then it holds that
 $\phi(x, \alpha_1\e_1+\alpha_2\e_2)=\alpha_1\phi(x, \e_1)+\alpha_2\phi(x, \e_2)$.
 
We define the \emph{Choquet boundary} $\Ch_{\H} K$ of $\H$ as the Choquet boundary of $\A$, that is, $\Ch_{\H} K$ is the set of those points $x \in K$ such that $i(x)$ is an extreme point of $B_{\A^{*}}$.

From now on, let $K$ be compact. 

 The symbol $\ep_x$ stands for the Dirac measure at the point $x\in K$. If $f \colon K \rightarrow \ef$ is a bounded Borel function and $\mu \in \M(K, \ef)$ then the symbol $\mu(f)$ stands for $\int_K f d\mu$.
 
 Let us now recall some notions of the theory of vector measures. Let $\mu$ be a set function defined on the Borel sets of $K$ with values in $E^*$. Then $\mu$ is called \emph{completely aditive}, if for every sequence of pairwise disjoint Borel sets $\lbrace A_i \rbrace_{i=1}^\infty$ it holds that
 \[ \mu(\bigcup_{i=1}^\infty A_i)=\sum_{i=1}^\infty \mu(A_i) \]
 in the sense of norm convergence in $E^*$. If $\mu$ is completely aditive, then it is called \emph{regular}, if for every $e \in E$ the scalar measure $\langle \mu, e \rangle$ defined by 
\[ \la \mu, e \ra(A)=\la \mu(A), e \ra, \quad A \subset K \text{ Borel},\]
 is both inner and outer regular in the usual sense. 
The variation  of $\mu$ is a measure defined by
\[\abs{\mu}(A)=\sup \{\sum_{i=1}^{n} \norm{\mu(A_i)}: \{ A_i \}_{i=1}^{n} \text{ is a Borel partition of } A\}, \quad A \subset K \text{ Borel}.\]
Finally, $\mu \in \M(K, \E)$ if $\mu$ is completely aditive, regular and $\abs{\mu}(K) < \infty$. The norm of $\mu \in \M(K, \E)$ is $\norm{\mu}=\abs{\mu}(K)$. It follows from the definitions that if $\mu \in \M(K, \E)$, then for every $e \in E$, the scalar measure $\la \mu, e \ra$ belongs to $\M(K, \ef)$.

 For $f \colon K \rightarrow \ef$ and $e \in E$, the function $f \otimes e \colon K \rightarrow E$ is defined by
 \[ (f \otimes e)(x)=f(x)e, \quad x \in K .\]
 If $f \in \C(K, \ef)$ and $e \in E$, then $f \otimes e \in \C(K, E)$ with $\norm{f \otimes e}=\norm{f}\norm{e}$, and it follows from the form of duality between $\M(K, \E)$ and $\C(K, E)^*$ (see \cite[pages 192 and 193]{Singer}) that 
  \begin{equation}
 \label{application}
 \la \mu, e \ra(f)=\langle \mu, f \otimes e \rangle, \quad \mu \in \M(K, \E).
 \end{equation}
 Also if $f \colon K \rightarrow \ef$ is a bounded Borel function, then for a vector measure $\mu \in \M(K, \E)$ and $e \in E$ we consider the application $\la \mu, f \otimes e \ra$ of $\mu$ on $f \otimes e$ given by \eqref{application}.
 
 Further, if $\mu \in \M(K, \ef)$ and $\e \in \E$, then the vector measure $\e \mu \in \M(K, \E)$ is defined by
 \[  \la \e \mu, h \ra= \mu(\e(h)), \quad h \in \C(K, E).\]

Later we will frequently use the fact that for a bounded Borel function $f \colon K \rightarrow \ef$, $\mu \in \M(K, \ef)$, $\e \in \E$ and $e \in E$ it holds that
\begin{equation}
\label{aplikace}
\la \e \mu, f \otimes e \ra=\la \e, e \ra \mu(f).
\end{equation}
By \eqref{application} we have $\la \e \mu, f \otimes e \ra=\la \e \mu, e \ra(f)$. Thus to prove \eqref{aplikace} it is enough to show that the scalar measures $\la \e, e \ra \mu$ and  $\la \e \mu, e \ra$ coincide in $\M(K, \ef)$. To this end, let $h \in \C(K, \ef)$. Then 
\begin{equation}
\nonumber
\la \e \mu, e \ra(h)=\la \e\mu, h \otimes e\ra=\mu(\e(h \otimes e))=\mu(\la \e, e \ra h)=\la \e, e \ra \mu(h).
\end{equation}
Thus $\la \e \mu, e \ra=\la \e, e \ra \mu$, and \eqref{aplikace} holds.

If $x \in K$, then each $\mu \in \M(K, \E)$ can be uniquely decomposed as $\mu=\psi \ep_x+\nu$, where $\psi \in \E$ and $\nu \in \M(K, \E)$ with $\nu(\{x\})=0$. To see this, first observe that $\mu|_{\{x\}}=\mu(\{x\})\ep_x $. Indeed, for a Borel set $A \subset K$ and $e \in E$ we have
\[\la \mu|_{\{x\}}(A), e \ra=\la \mu(A \cap \{x\}), e \ra=\la \mu(\{x\})\ep_x(A) , e \ra.\]
Thus if we denote $\psi=\mu(\{x\})$ and $\nu=\mu|_{K \setminus \{x\}}$, then 
\[\mu=\mu|_{\{x\}}+\mu|_{K \setminus \{x\}}=\psi\ep_x+\nu.\]
The uniqueness part is easy. Whenever we write a vector measure $\mu \in \M(K, \E)$ in the form $\mu=\psi \ep_x+\nu$, then we implicitly mean that $\psi \in \E$ and $\nu(\{x\})=0$.

The methods of \cite{cambern-pacific} heavily rely on the description of the second dual space of $\C(K, E)$. It was first shown by Kakutani \cite{kakutani} that the second dual space of $\C(K, \ef)$ is in the form $\C(Z, \ef)$, where $Z$ is a compact Hausdorff space depending on $K$. Moreover, there exists a natural mapping $t:K \rightarrow Z$ which maps $K$ onto the set of isolated points of $Z$, see \cite[Corollary 4.2]{daleslaustrauss}. 

Further, it was shown in \cite{camberngriem2} that if $E$ is a Banach space such that $\E$ has the Radon-Nikodym property (in particular, if $E$ is reflexive), then the space $\C(K, E)^{**}$ is isometrically isomorphic to the space $\C(Z, E^{**}_{\sigma^\ast})$, where $Z$ is the compact Hausdorff space satisfying $\C(K, \ef)^{**} \simeq \C(Z, \ef)$, and $E^{**}_{\sigma^\ast}$ denotes $E^{**}$ equipped with its weak$^*$ topology. The proof of this fact consists of the following series of isometries:
\begin{equation}
\nonumber
\begin{aligned}
&\C(K, E)^{**} \simeq \M(K, \E)^*\simeq \left[ \M(K, \ef) \hat{\otimes} \E \right]^* \simeq \left[ \E \hat{\otimes} \M(K, \ef) \right]^* \simeq\\&\simeq
 L(\E, \M(K, \ef)^*) \simeq L(\E, \C(Z, \ef)) \simeq \C(Z, E^{**}_{\sigma^*}),
\end{aligned}
\end{equation}
where the last isometry is defined for $F \in L(\E, \C(Z, \ef)) \simeq \C(Z, E^{**}_{\sigma^*})$ by the equality
\begin{equation}
\label{pom2}
\la F(z), \e \ra_{E^{**}, E^*}=F(\e)(z), \quad z \in Z, \e \in \E,
\end{equation}
see \cite[Theorem 2]{camberngriem2}.

Now suppose that $E$ is reflexive. For $F \in \C(K, E)^{**} \simeq \C(Z, E_{\sigma^\ast})$ and $\e \in \E$, the function $\e(F)$ defined for $z \in Z$ by $\e(F)(z)=\la \e, F(z) \ra$ belongs to $\C(Z, \ef)$, and by \eqref{pom2} it coincides with $F(\e)$, if $F$ is considered as the element of $L(\E, \C(Z, \ef))$.

 The following equality is frequently used by Cambern in \cite{cambern-pacific}:
\begin{equation}
\label{natural}
\begin{aligned}
\la F, \e \ep_{x} \ra_{\C(K, E)^{**}, \M(K, E^*)}=\la \ep_{tx}, \e(F) \ra_{\M(Z, \ef), \C(Z, \ef)}, \quad x \in K, \e \in \E,
\end{aligned}
\end{equation}
see page 252 in \cite{cambern-pacific}. 

For the sake of clarity we collect some of the results and arguments justifying \eqref{natural}. The scalar version of this equality follows from the work of Gordon \cite{gordon}, and was used by Cohen in \cite{cohen} to give a different proof of the theorem of Amir and Cambern. More specifically, Gordon proves that the action of $\Phi \in \C(Z, \ef) \simeq \C(K, \ef)^{**}$  on the Dirac measure $\ep_x \in \M(K, \ef)$ is given by 
\begin{equation}
\label{scalar}
\la \Phi, \ep_x \ra_{\C(K, \ef)^{**}, \M(K, \ef)} =\la \ep_{tx}, \Phi \ra_{\M(Z, \ef), \C(Z, \ef)}.
\end{equation}
 
Now, let $F \in \C(K, E)^{**} \simeq \left[ \E \hat{\otimes} \M(K, \ef) \right]^* \simeq L(\E, \M(K, \ef)^*) \simeq \C(Z, E_{\sigma^\ast})$, $x \in K$ and $\e \in \E$. Notice that the vector measure $\e\ep_x$ is in the setting of tensor products nothing else then the canonical tensor $\e \otimes \ep_x$. Thus by the form of the correspondence between $\left[\E \hat{\otimes} \M(K, \ef) \right]^*$ and $ L(\E, \M(K, \ef)^*)$  (see e.g. \cite[page 230, Corollary 2]{diesteluhl}) one obtains that
\begin{equation}
\label{pom1}
\begin{aligned}
&\la F, \e \ep_x \ra_{\C(K, E)^{**}, \M(K, E^*)}=
\la F, \e \otimes \ep_x \ra_{\left[\E \hat{\otimes} \M(K, \ef) \right]^*, \E \hat{\otimes} \M(K, \ef)}
=\\&=
\la F(\e), \ep_x \ra_{\M(K, \ef)^*, \M(K, \ef)}
=\la F(\e), \ep_x \ra_{\C(K, \ef)^{**}, \M(K, \ef)}.
\end{aligned}
\end{equation}

Thus we have 
\begin{equation}
\nonumber
\begin{aligned}
&\la F, \e \ep_{x} \ra_{\C(K, E)^{**}, \M(K, E^*)}=^{\eqref{pom1}}
\la F(\e), \ep_x \ra_{\C(K, \ef)^{**}, \M(K, \ef)} =^{\eqref{scalar}}\\&=
\la \ep_{tx}, F(\e) \ra_{\M(Z, \ef), \C(Z, \ef)}
=^{\eqref{pom2}}
\la \ep_{tx}, \e(F) \ra_{\M(Z, \ef), \C(Z, \ef)}, 
\end{aligned}
\end{equation}
which verifies \eqref{natural}. 

It is worth mentioning that in \cite[Section 6]{cidralgalegovillamizar}, there is given a similar description of the second dual space of $\C_0(K, E)$, where $K$ is just locally compact. We will need to use only the case when $K$ is compact, though. 

Now we collect several lemmas needed for the proofs of the main results. We start with generalizations of lemmas used in \cite{rondos-spurny-spaces} to the vector valued context. Lemma \ref{comp} will allow us to consider compact spaces in the proofs of the main results instead of locally compact spaces.

\begin{lemma}
\label{extreme}
Let $\H$ be a closed subspace of $\C(K, E)$ for some compact space $K$ and a Banach space $E$ and let $\phi\colon K \times B_{\E} \to B_{\H^*}$ be the evaluation mapping. Then
\[
\ext B_{\H^*}\subset \phi(\Ch_{\H} K \times \ext B_{\E}).
\]
\end{lemma}

\begin{proof}
Let $\A$ be the canonical scalar function space associated to $\H$. For technical reasons we consider the space $BA=BA_{hom}(B_{\A^*} \times B_{\E}, \ef) \subset \C(B_{\A^*} \times B_{\E}, \ef)$ of all continuous functions from $B_{\A^*} \times B_{\E}$ to $\ef$ that are affine and homogenous in both variables, endowed with the supremum norm. The space $\H$ is isometrically embedded in $BA$ by the mapping $U: \H \rightarrow BA$ defined by \[Uh(a^*, \e)=\la a^*, \e(h) \ra, \quad h \in \H, a^* \in B_{\A^*}, \e \in B_{\E}.\]
We consider the weak$^*$ topology on the dual unit ball $B_{BA^*}$.

First we show that $\ext B_{BA^*} \subset \{\ep_{(a^*, \e)}|_{BA}: a^* \in B_{\A^*}, \e \in B_{\E}\}$. Since the latter set is compact, by the Milman theorem it is enough to show that 
\[\ov{co}(\{\ep_{(a^*, \e)}|_{BA}: a^* \in B_{\A^*}, \e \in B_{\E}\})=B_{BA^*}.\]
Assuming the contrary, there exist 
\[s \in B_{BA^*} \setminus \ov{co}(\{\ep_{(a^*, \e)}|_{BA}: a^* \in B_{\A^*}, \e \in B_{\E}\}),\text{ } \alpha \in \er \text{ and }f \in BA\]
 such that 
\begin{equation}
\nonumber
\begin{aligned}
&\re \la s, f \ra>\alpha>\sup \lbrace \re f(a^*, \e): a^* \in B_{\A^*}, \e \in B_{\E}\}
=\\&=
\sup \lbrace \abs{f(a^*, \e)}: a^* \in B_{\A^*}, \e \in B_{\E}\}=\norm{f},
\end{aligned}
\end{equation}
by the homogenity of the function $f$. 
Thus we have
\begin{equation}
\nonumber
\begin{aligned}
&\alpha <\re \la s, f \ra
\leq \norm{s} \norm{f}\leq \alpha\norm{s} \leq \alpha.
\end{aligned}
\end{equation}
This contradiction proves that $\ext B_{BA^*} \subset \{\ep_{(a^*, \e)}|_{BA}: a^* \in B_{\A^*}, \e \in B_{\E}\}$. 

Now we show that $\ext B_{BA^*} \subset \{\ep_{(a^*, \e)}|_{BA}: a^* \in \ext B_{\A^*}, \e \in \ext B_{\E}\}$. Let $s \in \ext B_{BA^*}$. We know from above that $s$ is of the form $\ep_{(a^*, \e)}|_{BA}$ for some $a^* \in B_{\A^*}$ and $\e \in B_{\E}$. 
We show that $\e \neq 0$. If $\e=0$, then $s=0 \in \ext B_{BA^*}$, thus $B_{BA^*}=\{0\}$. This gives a contradiction with the fact that both $\A^*$ and $E^*$ are nonzero spaces.
Now we assume that there are distinct $a_1^*, a_2^*$ such that $a^*=\frac{1}{2}(a_1^*+a_2^*)$. Then 
\[ \ep_{(a^*, \e)}|_{BA}=\frac{1}{2}(\ep_{(a_1^*, \e)}|_{BA}+\ep_{(a_2^*, \e)}|_{BA}),\]
by the affinity of functions from $BA$.
Moreover, the points $\ep_{(a_i^*, \e)}|_{BA}$ are distinct for $i=1,2$. Indeed, since $a_1^* \neq a_2^*$, there exists $a \in \A$ such that $\la a_1^*, a \ra \neq \la a_2^*, a \ra$. Since $\e \neq 0$, there exists $e \in E$ such that $\la \e, e \ra \neq 0$. Now, the function $a \otimes e$ belongs to $BA$, and
\[(a \otimes e)(a_1^*, \e)=\la \e, e \ra \la a_1^*, a \ra \neq \la \e, e \ra \la a_2^*, a \ra=(a \otimes e)(a_2^*, \e).\]
We arrived at a contradiction with $s \in \ext B_{BA^*}$, and hence we obtained that $a^* \in \ext B_{\A^*}$. By observing that the roles of $B_{\A^*}$ and $B_{\E}$ are symmetric, we deduce that $a^* \neq 0$ and $\e \in \ext B_{\E}$. 

Now, the set of extreme points of $B_{\A^*}$ is contained in the set 
\[ \{ \lambda i(x): x \in \Ch_{\H} K, \lambda \in S_{\ef}) \},\]
where $i \colon K \rightarrow B_{\A^*}$ is the evaluation mapping, see \cite[Lemma 2.1]{rondos-spurny-spaces}. Thus if $a^* \in \ext B_{\A^*}$ and $\e \in \ext B_{\E}$, then there exist $\lambda \in S_{\ef}$ and $x \in \Ch_{\H} K$ such that $a^*=\lambda i(x)$, and then 
\[ \ep_{(a^*, \e)}|_{BA}=\ep_{(\lambda i(x), \e)}|_{BA}=\ep_{(i(x),\lambda \e)}|_{BA},\]
by the homogenity of functions from $BA$.
It follows that the set $\ext B_{BA^*}$ is contained in 
\begin{equation}
\nonumber
\begin{aligned}
&\{\ep_{(i(x), \lambda \e)}|_{BA}:x \in \Ch_{\H} K, \e \in \ext B_{\E}, \lambda \in S_{\ef}\}
=\\&=
\{\ep_{(i(x), \e)}|_{BA}:x \in \Ch_{\H} K, \e \in \ext B_{\E}\}.
\end{aligned}
\end{equation}

Now, let $r: B_{BA^*} \rightarrow B_{\H^*}$ be the restriction mapping. Then $r$ is a continuous affine surjection, thus for each $s \in \ext B_{\H^*}$, the set $r^{-1}(s)$ contains an extreme point of $B_{BA^*}$, see \cite[Proposition 2.72]{lmns} and \cite[page 37]{alfsen}. Thus there exist $x \in \Ch_{\H} K$ and $\e \in \ext B_{\E}$ such that $r(\ep_{(i(x), \e)}|_{BA})=s$. 

Let $h \in \H$. Then, since $r=U^*$ we obtain that

\begin{equation}
\nonumber
\begin{aligned}
&\la s, h \ra=\ep_{(i(x), \e)}(Uh)=\la i(x), \e(h) \ra=\la \e, h(x) \ra=\la \phi(x, \e), h \ra.
\end{aligned}
\end{equation}

Thus $s=\phi(x, \e)$, which finishes the proof.
\end{proof}

The following is a partial result on representing functionals on subspaces of $\C(K, E)$ by vector measures carried by the Choquet boundary. For more advanced results in this area we refer the reader to \cite{saab-tala} and \cite{batty-london}.

\begin{lemma}
\label{repre}
Let $\H$ be a subspace of $\C(K, E)$ for some compact space $K$ and a Banach space $E$. Then for any $s\in\H^*$ there exists a vector measure $\mu\in\M(\ov{\Ch_{\H} K}, \E)$ such that $\mu=s$ on $\H$ and $\norm{\mu}=\norm{s}$.
\end{lemma}

\begin{proof}
	Let $s\in \H^*$ be given. Let $\A$ be the canonical scalar function space of $\H$.
We write $\B$ for the space $\{h|_{\ov{\Ch_{\H} K}}: h \in \H \} \subset \C(\ov{\Ch_{\H} K},E)$.
	We show that the restriction mapping $r\colon \H \to \B$ given by $r(h)=h|_{\ov{\Ch_{\H} K}}$ is an isometric isomorphism. Indeed, since for every $\e \in S_{\E}$ we have $\e(h) \in \A$, it follows by \cite[Theorem~2.3.8]{fleming-jamison-1} that 
	\[\sup_{x \in K} \norm{h(x)} =\sup_{\e \in S_{\E}, x \in K} \abs{\e (h)(x)}=
	\sup_{\e \in S_{\E}, x \in \ov{\Ch_{\H} K}} \abs{\e (h)(x)}=\sup_{x \in \ov{\Ch_{\H} K}} \norm{h(x)}.\]
	Thus one can define $t\in\B^*$ by the formula
	\[
	\la t, a \ra=\la s, h \ra,\quad h\in\H\text{ satisfies }h|_{\ov{\Ch_{\H} K}}=a,\quad a\in \B.
	\]
	Then $\norm{t}=\norm{s}$. Using the Hahn-Banach theorem we find a measure
	\[
	\mu\in (\C(\ov{\Ch_{\H} K},E))^*=\M(\ov{\Ch_{\H} K}, \E)
	\]
	such that $\norm{\mu}=\norm{t}$ and $t=\mu$ on $\B$.
	Then $\norm{\mu}=\norm{s}$ and
	\[
    \la \mu, h \ra=\la t, h|_{\ov{\Ch_{\H} K}} \ra=\la s, h \ra,\quad h\in\H.
	\]
	This finishes the proof.
\end{proof}

The important topological notion is that of a function of the first Borel class. Thus we recall that, given a pair of topological spaces $K,L$, a function $f\colon K\to L$ is \emph{of the first Borel class} if $f^{-1}(U)$ is a countable union od differences of closed sets in $K$ for any $U\subset L$ open (see \cite{spurny-amh} or \cite[Definition~5.13]{lmns}). 

For a bounded Borel function $f\colon K\to \ef$ we define a function $\wh{f}:\M(K, \ef) \to \ef$ by $\wh{f}(\mu)=\int_K f d\mu$. 

\begin{lemma}
\label{first-borel}
Let $K$ be a compact space and $f\colon K\to \ef$ be a bounded function of the first Borel class. Then for any $e \in S_E$, the function $\wh{f} \otimes e\colon \M(K,\E  )\to \ef$ defined as
\[
(\wh{f} \otimes e)(\mu)=\la \mu, f \otimes e \ra,\quad \mu\in \M(K,\E),
\]
is of the first Borel class on any ball $rB_{\M(K,\E)}$, $r>0$.
\end{lemma}

\begin{proof}
For $\mu \in \M(K, \E)$ we have
\[(\wh{f} \otimes e)(\mu)=\la \mu, f \otimes e \ra=\langle \mu, e \rangle(f)=\wh{f}(\langle \mu, e \rangle).\]
Since the bounded function $f$ is of the first Borel class, the function $\wh{f}$ is of the first Borel class on any ball $rB_{\M(K, \ef)}$, $r>0$, see \cite[Lemma 2.4]{rondos-spurny-spaces}. Moreover, the mapping $\mu \mapsto \langle \mu, e \rangle$ is weak$^*$-weak$^*$ continuous, and does not increase norm. Hence the assertion follows. 	
\end{proof}

If $K$ is locally compact Hausdorff space and $\H$ is a closed subspace of $\C_0(K, E)$, then we say that a point $x \in K$ is a \emph{weak peak point} (with respect to $\H$) if for each $\ep \in (0, 1)$, $e \in E$ and a neighbourhood $U$ of $x$ there exists a function $h \in B_{\A}$ (where $\A$ is the canonical scalar function space associated to $\H$) such that $h(x)>1-\ep$, $\vert h \vert \leq \ep$ on $\Ch_{\H} K \setminus U$ and $h \otimes e \in \H$.

\begin{remarks}
\label{remarks}
\begin{itemize}
\item[(i)]	
If $E$ is the scalar field, then $\A=\H$ and the above definition coincides with the definition of the weak peak points of scalar subspaces.
\item[(ii)]If $K$ is a locally compact Hausdorff space and $\H=\C_0(K, E)$ then $\A=\C_0(K, \ef)$, $\Ch_{\H} K=K$ and by the Urysohn Lemma, each point of $K$ is a weak peak point.
\item[(iii)]If $X$ is a compact convex set in a locally convex space and $\H=\fra(X, E)$, the space of all affine $E$-valued continuous functions, then $\A=\fra(X, \ef)$, $\Ch_{\H} X=\ext X$ and a point $x \in \ext X$ is a weak peak point with respect to $\H$ if and only if it is a weak peak point in the sense of \cite[p. 73]{cc}.
\item[(iv)] More generally, if $\H$ is a closed subspace of $\C_0(K, E)$ such that $\e(h) \otimes e \in \H$ whenever $h \in \H$, $e \in E$ and $\e \in \E$, then the set of weak peak points of $\H$ coincides with the set of weak peak points of its canonical scalar function space $\A$.
\end{itemize}
\end{remarks}

Next we check that as in the scalar case, each weak peak point belongs to the Choquet boundary.

\begin{lemma}
\label{weakpeakchoquet}
Let $\H$ be a closed subspace of $\C(K, E)$ for some compact space $K$ and a Banach space $E$ and $x \in K$ be a weak peak point. Then $x \in \Ch_{\H} K$. 
\end{lemma} 
\begin{proof}
Suppose that $x$ is a weak peak point. Let $\A$ be the canonical scalar function space associated to $\H$ and suppose that $\mu \in \M(\ov{\Ch_{\H} K}, \ef)$ with $\norm{\mu}\leq 1$ is a measure $\A$-representing the point $x$ in the sense that
\[h(x)=\mu(h), \quad h \in \A.\]
We fix an arbitrary closed neighborhood $U$ of $x$ and $\varepsilon>0$. Then there is a function $h \in B_{\A}$ satisfying
$$h(x)>1-\ep \quad \text{and} \quad \vert h \vert < \ep \text{ on } \Ch_{\H} K \setminus U.$$
Since $h$ is continuous and $U$ is closed, it even holds that
$\vert h \vert \leq \ep$ on the set $\overline{\Ch_{\H} K} \setminus U \subset \overline{\Ch_{\H} K \setminus U}$.
So, we have
\begin{equation}
\nonumber
\begin{aligned}
&1-\ep < h(x)=\int_K h~d\mu \leq \int_K \vert h \vert ~d\abs{\mu} =\int_{\overline{\Ch_{\H} K}} \vert h \vert ~d\abs{\mu}
=\\&=
\int_{\overline{\Ch_{\H} K} \cap U} \vert h \vert ~d\abs{\mu} + \int_{\overline{\Ch_{\H} K} \setminus U} \vert h\vert ~d\abs{\mu} \leq \abs{\mu}(U)+\ep.
\end{aligned}
\end{equation}

In other words, $\abs{\mu}(U)>1-2\ep$. Since $\ep>0$ is chosen arbitrarily, we have that $\abs{\mu}(U)=1$. Hence $\abs{\mu}(V)=1$ for each closed neighborhood $V$ of $x$. Since $\norm{\mu}\leq 1$, from this it easily follows that $\mu=\lambda\ep_x$ for some $\lambda \in S_{\ef}$. Thus
\[h(x)=\mu(h)=\lambda h(x), \quad h \in \A.\]
Since $x$ is a weak peak point, there exists a function $h \in \A$ such that $h(x) \neq 0$, which implies that $\lambda=1$.

We claim that from here it follows that $x$ belongs to the Choquet boundary of $K$. Indeed, suppose that $i(x)=\frac{1}{2}s_1+\frac{1}{2}s_2$ for some $s_1, s_2 \in B_{\A^*}$, where $i:K \rightarrow B_{\A*}$ is the evaluation mapping. By the scalar version of Lemma \ref{repre} there exist measures $\mu_1, \mu_2 \in \M(\ov{\Ch_{\H} K}, \ef)$ extending $s_1$ and $s_2$ respectively on $\A$ with $\norm{\mu_i}=\norm{s_i}$ for $i=1, 2$. Then $\mu=\frac{1}{2}\mu_1+\frac{1}{2}\mu_2$ $\A$-represents the point $x$. Thus we know from above that $\mu=\ep_x$. Since the set of Dirac measures is contained in the set of extreme points of $B_{\M(K, \ef)}$, it follows that also $\mu_1$ and $\mu_2$ are equal to $\ep_x$. From this it follows that $s_1=s_2=i(x)$.
\end{proof}

For a closed subspace $\H \subset \C(K, E)$ we write $\H^\perp$ for the set of all vector measures $\mu \in \M(K, \E)$ that are identically zero on $\H$.
 
\begin{lemma}\label{chix}
Let $\H$ be a closed subspace of $\C(K, E)$ for some compact space $K$ and a Banach space $E$. Let $x\in K$ be a weak peak point. Then for any
\[
\mu\in \M(\ov{\Ch_{\H} K}, \E)\cap \H^\perp
\]
holds 
\[\la \mu, \chi_{\{x\}} \otimes e \ra=0, \quad e \in S_E.\]
\end{lemma}

\begin{proof}
Let $\mu \in \M(\ov{\Ch_{\H} K}, \E)\cap \H^\perp, e \in S_E$ be arbitrary and $\ep \in (0,1)$ be given. We write $\mu=\psi\ep_x+\nu$, where $\psi \in \E$ and $\nu(\lbrace x \rbrace)=0$. Let $U$ be a closed neighbourhood of $x$ such that $\vert \langle \nu, e \rangle \vert(U) \leq \ep$. We find $h \in B_{\A}$, where $\A$ is the canonical scalar function space of $\H$, such that $h(x)>1-\ep$, $\vert h \vert \leq \ep$ on $\Ch_{\H} K \setminus U$ and $h \otimes e \in \H$. Then $\abs{h}\le \ep$ on
\[
\ov{\Ch_{\H} K}\setminus U\subset \ov{\Ch_{\H} K\setminus U},
\]
and thus
\[
\begin{aligned}
\abs{\la \mu, \chi_{\{x\}} \otimes e \ra}&=\abs{\la \mu, \chi_{\{x\}} \otimes e \ra-\la \mu, h \otimes e \ra}\le\abs{\psi\ep_x(\chi_{\{x\}}\otimes e-h\otimes e)}+\abs{\la \nu, h\otimes e \ra}\\
&\le \abs{\la \psi, e \ra}(1-h(x))+\int_{\ov{\Ch_{\H} K}\cap U} \abs{h}\di\abs{\langle \nu, e \rangle}+\int_{\ov{\Ch_{\H} K}\setminus U} \abs{h}\di\abs{\langle \nu, e \rangle}\\
&\le \Vert \psi \Vert \ep+\ep+\ep\norm{\langle \nu, e \rangle}\le \ep(1+\norm{\mu}).
\end{aligned}
\]
Hence $\la \mu, \chi_{\{x\}} \otimes e \ra=0$. 
\end{proof}

\begin{lemma}\label{faktor}
Let $\H$ be a closed subspace of $\C(K,E)$ for some compact space $K$ and a Banach space $E$ and let $\pi\colon \M(K, \E)\to \H^*$ be the restriction mapping. Let $x\in K$ be a weak peak point and $e \in S_E$. For each $\mu\in \M(K,\E)$ we define $(\wh{\chi_{\{x\}}} \otimes e)(\mu)=\la \mu, \chi_{\{x\}} \otimes e \ra$.  Then there exists $a_{x, e}^{**}\in \H^{**}$ such that \[\la a_{x,e}^{**}, \pi(\mu) \ra=(\wh{\chi_{\{x\}}} \otimes e)(\mu)\] for any measure $\mu \in \M(K, \E)$ carried by $\ov{\Ch_{\H} K}$.

Moreover, if $x_1$ and $x_2$ are distinct weak peak points in $K$, $e_1, e_2 \in S_E$ and $\alpha_1, \alpha_2 \in \ef$ are arbitrary, then
$\norm{\alpha_1 a_{x_1,e_1}^{\ast\ast}+\alpha_2 a_{x_2,e_2}^{\ast\ast}} = \max \lbrace \abs{\alpha_1}, \abs{\alpha_2} \rbrace$.
\end{lemma}

\begin{proof}
The element $\wh{\chi_{\{x\}}} \otimes e$ is contained in $(\M(K, \E))^*$. In order to find the required element $a_{x,e}^{**}\in \H^{**}$ it is enough to realize that for any $\mu\in \M(\ov{\Ch_{\H} K},\E)\cap \H^\perp$ we have $\la \mu, \chi_{\{x\}} \otimes e \ra=0$ (see Lemma~\ref{chix}). Thus for $s\in \H^*$ we can define
\[
\la a_{x,e}^{**}, s \ra=(\wh{\chi_{\{x\}}} \otimes e)(\mu),
\]
where $\mu$ is an arbitrary measure in $\pi^{-1}(s)$ carried by $\ov{\Ch_{\H} K}$ (see Lemma \ref{repre}).

Now suppose that $x_1$ and $x_2$ are distinct weak peak points in $K$, $e_1, e_2 \in S_E$ and $\alpha_1, \alpha_2 \in \ef$. If $s \in \H^\ast$, then by Lemma \ref{repre} there exists $\mu \in \pi^{-1}(s) \cap \M(\ov{\Ch_{\H} K}, \E)$ with $\norm{\mu}=\norm{s}$, and we have
\begin{equation}
\nonumber
\begin{aligned}
&\abs{\la \alpha_1 a_{x_1,e_1}^{\ast\ast}+\alpha_2 a_{x_2,e_2}^{\ast\ast}, s \ra}=\abs{(\alpha_1 (\wh{\chi_{\{x_1\}}} \otimes e_1)+\alpha_2( \wh{\chi_{\{x_2\}}} \otimes e_2))(\mu)}
=\\&=
\abs{\la \mu, \alpha_1 (\chi_{\{x_1\}} \otimes e_1)+\alpha_2 (\chi_{\{x_2\}} \otimes e_2) \ra}=\abs{\alpha_1 \la \mu(\lbrace x_1 \rbrace), e_1 \ra+\alpha_2\la \mu(\lbrace x_2 \rbrace), e_2 \ra}
\leq \\& \leq 
\abs{\alpha_1}\abs{\la \mu(\lbrace x_1 \rbrace), e_1 \ra}+\abs{\alpha_2}\abs{\la \mu(\lbrace x_2 \rbrace), e_2 \ra} 
\leq \\& \leq
 \max \lbrace \abs{\alpha_1}, \abs{\alpha_2} \rbrace (\norm{\mu (\lbrace x_1 \rbrace)} \norm{e_1}+\norm{\mu (\lbrace x_2 \rbrace)} \norm{e_2}) \leq \max \lbrace \abs{\alpha_1}, \abs{\alpha_2} \rbrace \norm{\mu}
 =\\&=
 \max \lbrace \abs{\alpha_1}, \abs{\alpha_2} \rbrace \norm{s}.
\end{aligned}
\end{equation}
To prove the reverse inequality, first observe that for $x \in \ov{\Ch_{\H} K}$, $\e \in B_{\E}$ and $h \in \H$ we have
\begin{equation}
\nonumber
\la \pi(\e\ep_x), h \ra=\la \e\ep_x, h \ra=\ep_x(\e(h))=\la \e, h(x) \ra=\la \phi(x, \e), h \ra.
\end{equation} 
Thus 
\begin{equation}
\label{projekt}
\pi(\e\ep_x)=\phi(x, \e)~\text{in}~\H^*.
\end{equation}
Now, suppose that $\abs{\alpha_1} \geq \abs{\alpha_2}$. There exists $\e \in S_{\E}$ such that $\la \e, e_1 \ra=1$. The measure $\e\ep_{x_1}$ is carried by $\ov{\Ch_{\H} K}$, thus by the definition of $a_{x_1, e_1}^{**}$, \eqref{aplikace} and \eqref{projekt} we have
\begin{equation}
\nonumber
\begin{aligned}
&\norm{\alpha_1 a_{x_1,e_1}^{\ast\ast}+\alpha_2 a_{x_2,e_2}^{\ast\ast}} \geq \abs{\la \alpha_1 a_{x_1,e_1}^{\ast\ast}+\alpha_2 a_{x_2,e_2}^{\ast\ast}, \phi(x_1, \e)\ra}
=\\&=
\abs{(\alpha_1 (\wh{\chi_{\{x_1\}}} \otimes e_1)+\alpha_2( \wh{\chi_{\{x_2\}}} \otimes e_2))(\e\ep_{x_1})}=\abs{\alpha_1}\abs{\la \e, e_1 \ra}=\abs{\alpha_1}.
\end{aligned}
\end{equation}
The proof is finished.
\end{proof}

\begin{lemma}
\label{pi}
Let $\H$ be a closed subspace of $\C(K, E)$ for some compact space $K$ and a Banach space $E$ and $\pi\colon \M(K,\E)\to \H^*$ be the restriction mapping.
Let $\wh{f}\in \M(K,\E)^*$ and $a^{**}\in \H^{**}$ satisfy $\la \wh{f}, \mu \ra=\la a^{**},\pi(\mu) \ra$ for $\mu \in \M(\ov{\Ch_{\H} K}, \E)$.
\begin{itemize}
\item [(a)] Then for any $s\in \H^*$ and $\mu\in \pi^{-1}(s) \cap \M(\ov{\Ch_{\H} K}, \E)$ holds
\[
\la a^{**}, s\ra=\la \wh{f}, \mu \ra.
\]
\item [(b)] For any $r>0$, if $\wh{f}$ is of the first Borel class on $rB_{\M(\ov{\Ch_{\H} K},\E)}$, then $a^{**}$ is of the first Borel class on $rB_{\H^*}$.
\end{itemize}
\end{lemma}

\begin{proof}
(a) Given $s\in\H^*$ and $\mu\in\pi^{-1}(s) \cap \M(\ov{\Ch_{\H} K}, \E)$, we have
\[
\la a^{**}, s \ra=\la a^{**}, \pi(\mu)\ra=\la \wh{f}, \mu\ra.
\]

(b) For any $r>0$, the mapping $\pi\colon rB_{\M(\ov{\Ch_{\H} K},\E)}\to rB_{\H^*}$ is a weak$^*$-weak$^*$ continuous surjection (see Lemma~\ref{repre}). By \cite[Theorem 10]{HoSp} (see also \cite[Theorem~5.26(d)]{lmns}), if $\wh{f}$ is of the first Borel class on $rB_{\M(\ov{\Ch_{\H} K},\E)}$, $a^{**}$ is of the first Borel class on $rB_{\H^*}$.
\end{proof}

\begin{lemma}
\label{minim}
Let $f\colon X\to \ef$ be an affine function of the first Borel class on a compact convex set $X$. Then
\[
\sup_{x\in X} \abs{f(x)}=\sup_{x\in \ext X}\abs{f(x)}.
\]
\end{lemma}

\begin{proof}
The assertion follows from \cite[Corollary~1.5]{dosp} since any function of the first Borel class has the point of continuity property (see \cite[Theorem~2.3]{koumou}).
\end{proof}

\begin{lemma}
	\label{comp}
	Let $\H$ be a closed subspace of $\C_0(K, E)$ for some locally compact space $K$ and a Banach space $E$. Let $J=K\cup\{\alpha\}$ be the one-point compactification of $K$, where $\alpha$ is the point at infinity.
	Let
	\[
	\widetilde{\H}=\{h\in \C(J,E)\colon h|_{K}\in \H\ \&\ h(\alpha)=0\}.
	\]
	Then $\widetilde{\H}$ is a closed subspace of $\C(J,E)$ isometric to $\H$ such that $\Ch_{\H} K$ is homeomorphic to $\Ch_{\widetilde{\H}} J$ and a point $x \in \Ch_{\H} K$ is a weak peak point with respect to $\H$ if and only if it is a weak peak point with respect to $\widetilde{\H}$.
\end{lemma}

\begin{proof}
	Clearly, any function $h\in\H$ has the unique extension $f_h\in \widetilde{\H}$ and the mapping $h\mapsto f_h$ is an isometric isomorphism. Thus $\widetilde{\H}$ is a closed subspace of $\C(J,E)$.
	
	If $\A$ is the canonical scalar function space of $\H$, then clearly, the canonical scalar function space of $\widetilde{\H}$ is of the form
	\[\widetilde{\A}=\lbrace g \in \C(J, \ef): g|_{K}\in \A\ \&\ g(\alpha)=0\}.\]
	It was proved in \cite[Lemma 2.8]{rondos-spurny-spaces} that the spaces $\A$ and $\widetilde{\A}$ are isometric and Choquet boundaries of $\A$ and $\widetilde{\A}$ are homeomorphic. It is now clear that the respective weak peak points in $\Ch_{\H} K$ and $\Ch_{\widetilde{\H}} J$ coincide. This finishes the proof.
\end{proof}

The next lemma describes a property of the parameter $\lambda$ that is crucial for the proof of the main theorem of \cite{cidralgalegovillamizar}.

\begin{lemma}
	\label{lambda}
	Let $E$ be a Banach space. Let $r \in \en$ and $\eta>0$ be fixed and suppose that $\lbrace e_i \rbrace_{i=1}^{2^r} \subset E$ satisfy $\norm{e_i} \geq \eta$ for each $1 \leq i \leq 2^r$. Then there exist $\lbrace \alpha_i \rbrace_{i=1}^{2^r} \subset \ef$ with $\max \lbrace \abs{\alpha_i}: 1 \leq i \leq 2^r \rbrace \leq 1$ such that
	\[\norm{\sum_{i=1}^{2^r} \alpha_i e_i} \geq \eta \lambda(E)^r.\]
\end{lemma}

\section{Isomorphisms with a small bound}

\emph{Proof of Theorem \ref{main}}
We first assume that the spaces $K_1, K_2$ are compact.
Secondly, we suppose that there exists an $\ep>0$ such that $\Vert Tf \Vert \geq (1+\ep)\Vert f \Vert$ for $f \in \H_1$ and $\Vert T \Vert < \min \lbrace \lambda(E_1), \lambda(E_2) \rbrace$ (otherwise we replace $T$ by the isomorphism $(1+\ep)\norm{T^{-1}}T$). We fix $P$ such that $1<P<1+\ep$. Hence $T$ satisfies $\Vert Tf \Vert > P\norm{f}$ for $f \in \H_1, f \neq 0$. 

\emph{Claim 1.: For any $a^{**}\in \H_1^{**}\setminus \{0\}$ and $b^{**}\in \H_2^{**}\setminus \{0\}$ we have $\norm{T^{**}a^{**}}>P\norm{a^{**}}$ and $\norm{(T^{-1})^{**}b^{**}}>\frac{1}{\min \lbrace \lambda(E_1), \lambda(E_2) \rbrace}\norm{b^{**}}$.}

For the proof see \cite[Lemma~4.2]{rondos-spurny}.

For $i=1, 2$, the space $\C(K_i, E_i)^{**}$ is of the form $\C(Z_i, (E_i)_{\sigma^\ast})$, where $Z_i$ is a compact Hausdorff space depending on $K_i$, and $t_i$ denotes the natural mapping from $K_i$ into $Z_i$. We recall that for $f \in \C(Z_i, (E_i)_{\sigma^*})$ and $\e_i \in \E_i$, the function $\e_i(f)$ defined for $z \in Z_i$ as $\e_i(f)(z)=\la \e_i, f(z) \ra$ belongs to $\C(Z_i, \ef)$. 

Further, for $i=1,2$, let $\pi_i\colon \M(K_i,\E_i)\to \H_i^{*}$ be the restriction mapping, $\pi_i^\ast: \H_i^{\ast\ast} \to \M(K_i, \E_i)^* \simeq \C(K_i, E_i)^{**} \simeq \C(Z_i, (E_i)_{\sigma^*})$ be its adjoint mapping and let $\phi_i\colon K_i \times B_{\E_i} \to B_{\H_i^*}$ be the evaluation mapping. It follows from the Hahn-Banach theorem that the mappings $\pi_1^*, \pi_2^*$ are into isometries. By \eqref{projekt}, for each $\s \in B_{\E_1}$, $\t \in B_{\E_2}$, $x \in K_1$ and $y \in K_2$ it holds that
\begin{equation}
\label{projekce}
\pi_1(\s\ep_x)=\phi_1(x, \s)\quad\text{in } \H_1^*\quad \text{and}\quad \pi_2(\t\ep_y)=\phi_2(y, \t) \quad\text{in } \H_2^*.
\end{equation} 

For each $x\in \Ch_{\H_1} K_1$ we consider the function ${\chi}_{\{x\}}$. 
Let for $u \in S_{E_1}$, $\wh{\chi_{\{x\}}} \otimes u \colon \M(K_1, \E_1)\to \ef$ be defined as in Lemma~\ref{first-borel} and let $a_{x, u}^{**}\in \H_1^{**}$ satisfy 
\[\la a_{x, u}^{**}, \pi_1(\mu) \ra=(\wh{\chi_{\{x\}}} \otimes u)(\mu)\]
 for $\mu$ carried by $\ov{\Ch_{\H_1} K_1}$ (see Lemma~\ref{faktor}). Then $a_{x, u}^{**}$ is of the first Borel class on $rB_{\H_1^*}$ for any $r>0$, see Lemma~\ref{first-borel} and~\ref{pi}(b).
Analogously we define for $y\in\Ch_{\H_2} K_2$ and $v \in S_{E_2}$ the function $\wh{\chi_{\{y\}}} \otimes v$ and the element $b_{y, v}^{**}\in \H_2^{**}$.  

We start the proof with the following series of equalities for $x \in \Ch_{\H_1} K_1, y \in \Ch_{\H_2} K_2$, $u \in S_{E_1}$ and $\t \in S_{\E_2}$:

\begin{equation}
\nonumber
\begin{aligned}
&\langle \t, \pi_2^\ast(T^{\ast\ast}(a_{x,u}^{\ast\ast}))(t_2y) \rangle_{\E_2, E_2}
=\t(\pi_2^*(T^{\ast\ast}(a_{x,u}^{\ast\ast})))(t_2y)
=\\&=\la \ep_{t_2y},\t(\pi_2^*(T^{\ast\ast}(a_{x,u}^{\ast\ast}))) \ra_{\M(Z_2, \ef), \C(Z_2, \ef)}
=\\&=^{\eqref{natural}}
\langle \pi_2^\ast(T^{\ast\ast}(a_{x,u}^{\ast\ast})), \t \ep_y \rangle_{\C(K_2, E_2)^{\ast\ast}, \M(K_2, E^*_2)}
=\\&=
\langle T^{\ast\ast}(a_{x,u}^{\ast\ast}), \pi_2(\t \ep_y) \rangle_{\H_2^{\ast\ast},\H_2^{\ast}}
=^{\eqref{projekce}}
\langle T^{\ast\ast}(a_{x,u}^{\ast\ast}), \phi_2(y, \t) \rangle_{\H_2^{\ast\ast},\H_2^{\ast}}.
\end{aligned}
\end{equation}

Suppose that $\mu \in \pi_1^{-1}(T^\ast(\phi_2(y, \t)))$ is a Hahn-Banach extension of $T^\ast(\phi_2(y, \t))$ carried by $\ov{\Ch_{\H_1} K_1}$ written in the form $\mu=\psi \ep_{x}+\nu$. Then we have
\begin{equation}
\nonumber
\begin{aligned}
&\langle T^{\ast\ast}(a_{x,u}^{\ast\ast}), \phi_2(y, \t) \rangle_{\H_2^{\ast\ast},\H_2^{\ast}}=\langle a_{x,u}^{\ast\ast}, T^\ast(\phi_2(y, \t)) \rangle_{\H_1^{**},\H_1^{\ast}}
=\\&=
\langle a_{x,u}^{\ast\ast}, \pi_1(\mu) \rangle_{\H_1^{**},\H_1^{\ast}}=
\langle \wh{\chi_{\{x\}}} \otimes u, \mu \rangle_{\C(K_1, E_1)^{\ast\ast},\M(K_1, E_1^*)} 
=\\&=
\langle \wh{\chi_{\{x\}}} \otimes u, \psi\ep_x+\nu \rangle_{\C(K_1, E_1)^{\ast\ast},\M(K_1, E_1^*)} =\langle \psi, u \rangle_{E_1^*, E_1}=
\langle \mu(\{x\}), u \rangle_{E_1^*, E_1}.
\end{aligned}
\end{equation}

Thus using the above notation, we have
\begin{equation}
\label{duality}
\begin{aligned}
&\la\t, \pi_2^\ast(T^{\ast\ast}(a_{x,u}^{\ast\ast}))(t_2y) \rangle_{\E_2, E_2}=\langle \pi_2^\ast(T^{\ast\ast}(a_{x,u}^{\ast\ast})), \t \ep_y \rangle_{\C(K_2, E_2)^{\ast\ast},\M(K_2, E_2^*)}
=\\&=
\langle T^{\ast\ast}(a_{x,u}^{\ast\ast}), \phi_2(y, \t) \rangle_{\H_2^{\ast\ast},\H_2^{\ast}}=\langle \psi, u \rangle_{\E_1, E_1}=\langle \mu(\{x\}), u \rangle_{\E_1, E_1}.
\end{aligned}
\end{equation}

Similarly, if $v \in S_{E_2}$, $\s \in S_{\E_1}$ and  $\mu \in \pi_2^{-1}((T^\ast)^{-1}(\phi_1(x, \s)))$ is a Hahn-Banach extension of
$(T^\ast)^{-1}(\phi_1(x, \s))$ carried by $\ov{\Ch_{\H_2} K_2}$ in the form $\mu=\psi \ep_{y}+\nu$, then we have the following:
\begin{equation}
\label{duality2}
\begin{aligned}
&\langle  \s, \pi_1^\ast((T^{-1})^{\ast\ast}(b_{y,v}^{\ast\ast}))(t_1x) \rangle_{\E_1, E_1}
=\\&=
\langle \pi_1^\ast((T^{-1})^{\ast\ast}(b_{y,v}^{\ast\ast})), \s \ep_x \rangle_{\C(K_1, E_1)^{\ast\ast},\M(K_1, E_1^*)}
=\\&=
\langle (T^{-1})^{\ast\ast}(b_{y,v}^{\ast\ast}), \phi_1(x, \s) \rangle_{\H_1^{\ast\ast},\H_1^{\ast}}=\la \psi, v \ra_{\E_2, E_2}=\langle \mu(\{y\}), v \rangle_{\E_2, E_2}.
\end{aligned}
\end{equation}
	
\begin{definition}
For $x \in \Ch_{\H_1} K_1$ and $y \in \Ch_{\H_2} K_2$ we define relations $\rho_1$ and $\rho_2$ as follows:
\begin{equation}
\label{rho}
\nonumber
\begin{aligned}
\rho_1(x)=&\{y\in \Ch_{\H_2} K_2, \exists v \in S_{E_2},\exists \s \in S_{\E_1} \colon \\&
\abs{ \langle (T^{\ast\ast})^{-1}(b_{y, v}^{\ast\ast}), \phi_1(x, \s) \rangle} >\frac{1}{\min \lbrace \lambda(E_1), \lambda(E_2) \rbrace} \},\\
\rho_2(y)=&\left\{x\in \Ch_{\H_1} K_1, \exists u \in S_{E_1},\exists \t \in S_{\E_2}: \abs{\langle T^{\ast\ast}(a_{x, u}^{\ast\ast}), \phi_2(y, \t) \rangle} >P\right\}.\\
\end{aligned}
\end{equation}
\end{definition}

In the rest of the proof we show that $\rho_1$ is the desired homeomorphism from $\Ch_{\H_1} K_1$ to $\Ch_{\H_2} K_2$, with $\rho_2$ being its inverse. 

First note that using \eqref{duality} and \eqref{duality2}, we have the following equivalent descriptions of the relations $\rho_1$ and $\rho_2$.
\begin{lemma}
	\label{ekvi}
	Let $x \in \Ch_{\H_1} K_1, y \in \Ch_{\H_2} K_2$. Then the following assertions are equivalent:
	\begin{itemize}
		\item[(i)] \[\rho_2(y)=x;\]
		\item[(ii)] \[\sup_{u \in S_{E_1}, \t \in S_{\E_2}} \abs{\la \t, \pi_2^\ast(T^{\ast\ast}(a_{x,u}^{\ast\ast}))(t_2y) \rangle}>P;\]
		\item[(iii)] there exists a point $\t \in S_{\E_2}$ such that whenever 
		\[\mu \in \pi_1^{-1}(T^\ast(\phi_2(y, \t))) \cap \M(K_1, \E_1)\]
		 is a Hahn-Banach extension of $T^\ast(\phi_2(y, \t))$ which is carried by $\ov{\Ch_{\H_1} K_1}$, then $\norm{\mu(\{x\})}>P$.
	\end{itemize}
\end{lemma}

\begin{proof}
It follows by \eqref{duality} that the assertions (i) and (ii) are equivalent and that (i) implies (iii). 	
On the other hand, if (iii) holds, then pick an arbitrary $\mu \in \pi_1^{-1}(T^\ast( \phi_2(y, \t)))$, a Hahn-Banach extension of $T^\ast( \phi_2(y, \t))$ carried by $\ov{\Ch_{\H_1} K_1}$ (such a $\mu$ exists by Lemma \ref{repre}). Then there exists a point $u \in S_{E_1}$ such that $\abs{\langle \mu(\{x\}), u \rangle}>P$. Then again by \eqref{duality}, 
\[\abs{\langle T^{\ast\ast}(a_{x, u}^{\ast\ast}), \phi_2(y, \t) \rangle}=\abs{\la \mu(\{x\}), u \rangle}>P,\] 
that is, $\rho_2(y)=x$.
\end{proof}

Similarly we have the following characterization.
\begin{lemma}
	\label{ekvi2}
	Let $x \in \Ch_{\H_1} K_1, y \in \Ch_{\H_2} K_2$. Then the following assertions are equivalent:
	\begin{itemize}
		\item[(i)] \[\rho_1(x)=y;\]
		\item[(ii)] \[\sup_{v \in S_{E_2}, \s \in S_{\E_1}} \abs{\la \s, \pi_1^\ast((T^{-1})^{\ast\ast}(b_{y,v}^{\ast\ast}))(t_1x) \ra}>(\min \lbrace \lambda(E_1), \lambda(E_2) \rbrace)^{-1};\]
		\item[(iii)] there exists a point $\s \in S_{\E_1}$ such that whenever 
		\[\mu \in \pi_2^{-1}((T^\ast)^{-1}(\phi_1(x, \s))) \cap \M(K_2, \E_2)\] is a Hahn-Banach extension of $(T^\ast)^{-1}(\phi_1(x, \s))$ which is carried by $\ov{\Ch_{\H_2} K_2}$, then $\norm{\mu(\{y\})}>(\min \lbrace \lambda(E_1), \lambda(E_2) \rbrace)^{-1}$.
	\end{itemize}
\end{lemma}

\emph{Claim 2. $\rho_1$ and $\rho_2$ are mappings.}

We show that $\rho_2(y)$ is at most single-valued for each $y\in \Ch_{\H_2} K_2$. Suppose that there are distinct $x_1, x_2 \in \Ch_{\H_1} K_1$ such that $\rho_2(y)=x_i$ for $i=1, 2$. By Lemma~\ref{ekvi}(ii) this means that there exist points $u_i \in S_{E_1}, i=1, 2$, such that \[\norm{\pi_2^\ast(T^{\ast\ast}(a_{x_i,u_i}^{\ast\ast}))(t_2y)}_{E_2}>P, \quad i=1, 2.\]

Then by Lemma \ref{lambda} there exist $\alpha_1, \alpha_2 \in \ef$ with $\vert \alpha_i \vert \leq 1$ for $i=1, 2$, such that
\[ \norm{\alpha_1\pi_2^\ast(T^{\ast\ast}(a_{x_1,u_1}^{\ast\ast}))(t_2y)+\alpha_2\pi_2^\ast(T^{\ast\ast}(a_{x_2,u_2}^{\ast\ast}))(t_2y)}_{E_2} \geq P \lambda(E_2).\]
Thus 
\begin{equation}
\nonumber
\begin{aligned}
&\norm{\pi_2^\ast(T^{\ast\ast}(\alpha_1a_{x_1,u_1}^{\ast\ast}+\alpha_2a_{x_2,u_2}^{\ast\ast}))}_{\sup} 
\geq \\& \geq
\norm{\alpha_1\pi_2^\ast(T^{\ast\ast}(a_{x_1,u_1}^{\ast\ast}))(t_2y)+\alpha_2\pi_2^\ast(T^{\ast\ast}(a_{x_2,u_2}^{\ast\ast}))(t_2y)}_{E_2} \geq P \lambda(E_2) > \lambda(E_2).
\end{aligned}
\end{equation}

On the other hand, the function $\alpha_1 a_{x_1, u_1}^{\ast\ast}+\alpha_2 a_{x_2, u_2}^{\ast\ast} \in \H_1^{\ast\ast}$ satisfies 
\[\norm{\alpha_1 a_{x_1, u_1}^{\ast\ast}+\alpha_2 a_{x_2, u_2}^{\ast\ast}} \leq 1\]
 by Lemma \ref{faktor}. Thus we have obtained a contradiction with $\norm{\pi_2^\ast(T^{\ast\ast})}=\norm{T^{\ast\ast}}<\min \lbrace \lambda(E_1), \lambda(E_2) \rbrace \leq \lambda(E_2)$,  and hence $\rho_2$ is a mapping.
Analogously we show that $\rho_1(x)$ is at most single-valued for each $x\in \Ch_{\H_1} K_1$.

Next we use the maximum principle to check that the mappings $\rho_1$ and $\rho_2$ are surjective.

\begin{lemma}
\label{principle}
The following assertions hold.
\begin{itemize}
\item[(i)] 
If $x \in \Ch_{\H_1} K_1$ and $u \in S_{E_1}$, then 
\begin{equation}
\nonumber
\begin{aligned}
&\norm{ T^{**}(a_{x, u}^{**})}=\sup_{y \in \Ch_{\H_2} K_2, \t \in S_{\E_2}} \abs{\langle T^{**}(a_{x, u}^{**}), \phi_2(y, \t) \rangle}
=\\&=
\sup_{y \in \Ch_{\H_2} K_2} \norm{\pi_2^*(T^{**}(a_{x, u}^{**}))(t_2y)}.
\end{aligned}
\end{equation}
\item[(ii)] 
If $y \in \Ch_{\H_2} K_2$ and $v \in S_{E_2}$, then
\begin{equation}
\nonumber
\begin{aligned}
&\norm{ (T^{-1})^{**}(b_{y, v}^{**})}=\sup_{x \in \Ch_{\H_1} K_1, \s \in S_{\E_1}} \abs{\langle (T^{-1})^{**}(b_{y, v}^{**}), \phi_1(x, \s) \rangle}
=\\&=
\sup_{x \in \Ch_{\H_1} K_1} \norm{\pi_1^*((T^{-1})^{**}(b_{y, v}^{**}))(t_1x)}.
\end{aligned}
\end{equation}
\end{itemize}
\end{lemma}

\begin{proof}
We prove (i), the proof of (ii) is similar. First we show that the element $T^{**}a_{x, u}^{**}\in \H_2^{**}$ is of the first Borel class on $B_{\H_2^*}$.

Indeed, we know that $a_{x, u}^{**}$ is of the first Borel class on any ball in $\H_1^*$, in particular on $2B_{\H_1^*}$. Since $T^*$ is a weak$^*$-weak$^*$ homeomorphism, $T^*(B_{\H_2^*})\subset 2B_{\H_1^*}$ and $T^{**}(a_{x, u}^{**})=a_{x, u}^{**}\circ T^*$, it follows that $T^{**}(a_{x, u}^{**})$ is of the first Borel class on $B_{\H_2^*}$ as well. Thus by Lemma~\ref{minim} and~\ref{extreme} we have
\begin{equation}
\nonumber
\begin{aligned}
&\Vert T^{**}(a_{x, u}^{**}) \Vert =\sup_{h^* \in B_{\H_2^*}} \abs{\langle T^{**}(a_{x, u}^{**}), h^* \rangle}
=\\&=
\sup_{h^* \in \ext B_{\H_2^*}} \abs{\langle T^{**}(a_{x, u}^{**}), h^* \rangle} = \sup_{y \in \Ch_{\H_2} K_2, \t \in S_{\E_2}} \abs{\langle T^{**}(a_{x, u}^{**}), \phi_2(y, \t) \rangle}
=^{\eqref{duality}}\\&=
\sup_{y \in \Ch_{\H_2} K_2, \t \in S_{\E_2}} \abs{\la \t,  \pi_2^\ast(T^{\ast\ast}(a_{x, u}^{\ast\ast}))(t_2y) \rangle}=\sup_{y \in \Ch_{\H_2} K_2} \norm{\pi_2^*(T^{**}(a_{x, u}^{**}))(t_2y)}.
\end{aligned}
\end{equation}	
\end{proof}

Let $L_1$ and $L_2$ denote the domain of $\rho_1$ and $\rho_2$, respectively.

\emph{Claim 3.: The mappings $\rho_1\colon L_1\to \Ch_{\H_2} K_2$ and $\rho_2\colon L_2\to \Ch_{\H_1} K_1$ are surjective.}
Let $x\in \Ch_{\H_1} K_1$ be given and choose arbitrary $u \in S_{E_1}$. By Lemma \ref{faktor} we know that $\norm{a_{x, u}^{**}}=1$. Thus by Lemma \ref{principle} we have

\[
\begin{aligned}
&P < \Vert T^{**}(a_{x, u}^{**}) \Vert = \sup_{y \in \Ch_{\H_2} K_2, \t \in S_{\E_2}} \abs{\langle T^{**}(a_{x, u}^{**}), \phi(y, \t) \rangle}.
\end{aligned}
\]
Thus there exist $y \in \Ch_{\H_2} K_2$ and $\t \in S_{\E_2}$ such that $P<\abs{\langle T^{**}(a_{x, u}^{**}), \phi(y, \t) \rangle}$, that is, $\rho_2(y)=x$.
Analogously we check that $\rho_1$ is surjective.

\emph{Claim 4.: We have $L_1=\Ch_{\H_1} K_1$ and $L_2=\Ch_{\H_2} K_2$ and $\rho_2(\rho_1(x))=x$, $x\in\Ch_{\H_1} K_1$, and $\rho_1(\rho_2(y))=y$, $y\in\Ch_{\H_2} K_2$.}

Suppose that $y \in \Ch_{\H_2} K_2$, $\rho_2(y)=x$, but $x \notin L_1$ or $\rho_1(x) \neq y$. In both cases we obtain by Lemma \ref{ekvi2}(ii) that for all $v \in S_{E_2}$, $\norm{\pi_1^*((T^{**})^{-1}(b_{y, v}^{**}))(t_1x)}\leq (\min \lbrace \lambda(E_1), \lambda(E_2) \rbrace)^{-1}$. For $v \in S_{E_2}$ we denote 
\[ Q_v=\sup_{\tilde{x} \in \Ch_{\H_1} K_1} \norm{\pi_1^*((T^{**})^{-1}(b_{y, v}^{**})(t_1\tilde{x})}=^{\text{Lemma } \ref{principle}\text{(ii)}}\norm{(T^{-1})^{**}(b_{y,v}^{**})}\]
and 
\[Q=\sup_{v \in S_{E_2}} Q_v.\]

We know that $\rho_1$ is surjective. This by Lemma \ref{ekvi2}(ii) means that \[Q >(\min \lbrace \lambda(E_1), \lambda(E_2) \rbrace)^{-1}.\] Let $\ep>0$ satisfy 

\[\ep<\frac{2P-\min \lbrace \lambda(E_1), \lambda(E_2) \rbrace}{\min \lbrace \lambda(E_1), \lambda(E_2) \rbrace P}\quad\text{and}\quad Q -\ep> (\min \lbrace \lambda(E_1), \lambda(E_2) \rbrace)^{-1}.\]

By the definition of $Q$, let $v \in S_{E_2}$ and $\tilde{x} \in \Ch_{\H_1} K_1$ be such that the vector $u_1=\pi_1^*((T^{**})^{-1}(b_{y, v}^{**})(t_1\tilde{x})$ satisfies $\norm{u_1}>Q -\ep$. Since $Q -\ep> (\min \lbrace \lambda(E_1), \lambda(E_2) \rbrace)^{-1}$, this by Lemma \ref{ekvi2}(ii) means that $\rho_1(\tilde{x})=y$. Thus $\tilde{x} \neq x$.
We denote $u_2=\frac{u_1}{\norm{u_1}} \in S_{E_1}$.

Now we consider the element $T^{**}(a_{\tilde{x}, u_2}^{**})$. Since $\norm{a_{\tilde{x}, u_2}^{**}}=1$ by Lemma \ref{faktor}, we know that \[\norm{T^{**}(a_{\tilde{x}, u_2}^{**})} > P.\] This by Lemma \ref{principle}(i) means that there exists $\tilde{y} \in \Ch_{\H_2} K_2$ such that 
\[\norm{\pi_2^*(T^{**}(a_{\tilde{x}, u_2}^{**})(t_2\tilde{y})}>P.\] 
Thus there exists $\t \in S_{\E_2}$ such that
\[\abs{\la \t, \pi_2^*(T^{**}(a_{\tilde{x}, u_2}^{**})(t_2\tilde{y}) \ra}>P.\] 
Hence $\rho_2(\tilde{y})=\tilde{x}$ by Lemma \ref{ekvi}(ii). Thus $y \neq \tilde{y}$, since $\rho_2(\tilde{y})=\tilde{x}\neq x =\rho_2(y)$. 
Now, if we pick $\mu \in \pi_1^{-1}(T^*(\phi_2(\tilde{y}, \t)))$, a Hahn-Banach extension of $T^*(\phi_2(\tilde{y}, \t))$ carried by $\ov{\Ch_{\H_1} K_1}$, and write it in the form $\mu=\psi \ep_{\tilde{x}}+\nu$, where $\psi \in \E_1$ and $\nu \in \M(\ov{\Ch_{\H_1} K_1}, \E)$ with $\nu(\{\tilde{x}\})=0$, then by \eqref{duality},
\[\abs{\langle \psi, u_2 \rangle}=\abs{\la \t, \pi_2^*(T^{**}(a_{\tilde{x}, u_2}^{**})(t_2\tilde{y}) \ra}>P.\]
Thus $\norm{\psi}>P$ and 
\[\abs{\langle \psi, u_1 \rangle}=\norm{u_1} \abs{\langle \psi, u_2 \rangle}>(Q-\ep) P.\]
Now we have
\begin{equation}
\nonumber
\begin{aligned}
0&=\la \wh{\chi_{y}} \otimes v, \t\ep_{\tilde{y}} \ra_{\C(K_2, E_2)^{**}, \M(K_2, \E_2)}=\la b_{y, v}^{**}, \pi_2(\t \ep_{\tilde{y}}) \ra_{\H_2^{**}, \H_2^*}
=^{\eqref{projekce}}\\&=
 \langle b_{y, v}^{**}, \phi_2(\tilde{y}, \t) \rangle_{\H_2^{**}, \H_2^*}=
\langle (T^{**})^{-1}b_{y, v}^{**}, T^*\phi_2(\tilde{y}, \t) \rangle_{\H_1^{**}, \H_1^*}
=\\&=
\langle (T^{**})^{-1}b_{y, v}^{**}, \pi_1(\mu) \rangle_{\H_1^{**}, \H_1^*}=
\langle \pi_1^*((T^{**})^{-1}b_{y, v}^{**}), \mu \rangle_{\C(K_1, E_1)^{**}, \M(K_1, \E_1)}
=\\&=
\langle \pi_1^*((T^{**})^{-1}b_{y, v}^{**}), \psi \ep_{\tilde{x}}+\nu \rangle_{\C(K_1, E_1)^{**}, \M(K_1, \E_1)}
=^{\eqref{duality2}}\\&=
\langle \psi,\pi_1^*((T^{**})^{-1}b_{y, v}^{**})(t_1\tilde{x}) \rangle_{\E_1, E_1}+\langle \pi_1^*((T^{**})^{-1}b_{y, v}^{**}),\nu \rangle_{\C(K_1, E_1)^{**}, \M(K_1, \E_1)}
=\\&=
\langle \psi, u_1 \rangle_{\E_1, E_1}+ \langle \pi_1^*((T^{**})^{-1}b_{y, v}^{**}),\nu \rangle_{\C(K_1, E_1)^{**}, \M(K_1, \E_1)}.
\end{aligned}
\end{equation}	
On the other hand, we know that $\abs{\langle \psi, u_1 \rangle} >(Q-\ep) P$ and $\norm{\nu} \leq \norm{\mu}-\norm{\psi}<\min \lbrace \lambda(E_1), \lambda(E_2) \rbrace-P$, and thus

\begin{equation}
\nonumber
\begin{aligned}
\abs{\langle \pi_1^*(T^{**})^{-1}b_{y, v}^{**},\nu \rangle} &\leq \norm{\pi_1^*(T^{**})^{-1}b_{y, v}^{**}}\norm{\nu}
\leq \norm{(T^{**})^{-1}b_{y, v}^{**}}(\norm{\mu}-\norm{\psi})
<\\&<
Q_{v}(\min \lbrace \lambda(E_1), \lambda(E_2) \rbrace-P) \leq Q(\min \lbrace \lambda(E_1), \lambda(E_2) \rbrace-P).
\end{aligned}
\end{equation}

Thus it is necessary that $(Q -\ep)P \leq Q(\min \lbrace \lambda(E_1), \lambda(E_2) \rbrace-P)$, that is, 
\[\ep \geq \frac{Q(2P-\min \lbrace \lambda(E_1), \lambda(E_2) \rbrace)}{P} \geq \frac{2P-\min \lbrace \lambda(E_1), \lambda(E_2) \rbrace}{\min \lbrace \lambda(E_1), \lambda(E_2) \rbrace P}.\]
This contradicts the choice of $\ep$ and shows that $x \in L_1$ and $\rho_1(x)=y$.

Now, let $x\in\Ch_{\H_1} K_1$ be given. Then there exists $y\in L_2$ such that $\rho_2(y)=x$. Then $y=\rho_1(\rho_2(y))=\rho_1(x)$, which means that $x\in L_1$.

Let $y\in\Ch_{\H_2} K_2$ be given. Then we can find $x\in L_1=\Ch_{\H_1} K_1$ with $\rho_1(x)=y$ and further we can select $\wh{y}\in L_2$ such that $\rho_2(\wh{y})=x$. Then
\[
y=\rho_1(x)=\rho_1(\rho_2(\wh{y}))=\wh{y}\in L_2.
\]
Hence $L_2=\Ch_{\H_2} K_2$.

Finally, if $x\in \Ch_{\H_1} K_1$, we find $y\in \Ch_{\H_2} K_2$ with $\rho_2(y)=x$ and obtain
\[
\rho_2(\rho_1(x))=\rho_2(\rho_1(\rho_2(y)))=\rho_2(y)=x.
\]

Till now we have proved that $\rho_1\colon \Ch_{\H_1} K_1\to \Ch_{\H_2} K_2$ is a bijection with $\rho_2$ being its inverse.
Now we check that $\rho_1$ is a homeomorphism. To this end, note that the definition of the mappings $\rho_1$ and $\rho_2$ may be now be rewritten in the following way:
\begin{equation}
\label{prepis}
\begin{aligned}
\rho_1(x)=&\{y\in \Ch_{\H_2} K_2, \forall v \in S_{E_2} \exists \s \in S_{\E_1} \colon\\&
 \abs{\langle (T^{\ast\ast})^{-1}(b_{y, v}^{\ast\ast}), \phi_1(x, \s) \rangle} >\frac{1}{\min \lbrace \lambda(E_1), \lambda(E_2) \rbrace} \}, \quad x \in \Ch_{\H_1} K_1,\\
\rho_2(y)&=\{x\in\Ch_{\H_1} K_1, \forall u \in S_{E_1} \exists \t \in S_{\E_2}\colon\\&
\abs{\langle T^{\ast\ast}(a_{x, u}^{\ast\ast}), \phi_2(y, \t) \rangle} >P\}, \quad y \in \Ch_{\H_2} K_2.\\
\end{aligned}
\end{equation}
We show that the formula above holds for $\rho_2$, the proof for $\rho_1$ is similar. Suppose that $y \in \Ch_{\H_2} K_2$ and $\rho_2(y)=x$. If $u \in S_{E_1}$ is arbitrary, then by Lemma \ref{principle} we obtain that
\[
\begin{aligned}
&P < \Vert T^{**}(a_{x, u}^{**}) \Vert = \sup_{\tilde{y} \in \Ch_{\H_2} K_2, \t \in S_{\E_2}} \abs{\langle T^{**}(a_{x, u}^{**}), \phi_2(\tilde{y}, \t) \rangle}.
\end{aligned}
\]
Thus there exist $\tilde{y} \in \Ch_{\H_2} K_2$ and $\t \in S_{\E_2}$ such that $\abs{\langle T^{**}(a_{x, u}^{**}), \phi_2(\tilde{y}, \t) \rangle}>P$. This means that $\rho_2(\tilde{y})=x$. But since we know that $\rho_2$ is a bijection, this means that $y=\tilde{y}$. 

\emph{Claim 5.: The mapping $\rho_2$ is continuous.}

Assuming the contrary, there exists a net $\lbrace y_\beta: \beta \in B \rbrace \subset \Ch_{\H_2} K_2$ such that $y_\beta \rightarrow y_0 \in \Ch_{\H_2} K_2$ but $x_\beta=\rho_2(y_\beta) \nrightarrow \rho_2(y_0)=x_0$. Then there exists a closed neighbourhood $V$ of $x_0$ such that for each $\beta_0 \in B$ there exists $\beta \geq \beta_0$ such that $x_\beta \notin V$. 

Fix a point $u \in S_{E_1}$. Since $\rho_2(y_0)=x_0$, by \eqref{prepis} and \eqref{duality} there exists $\t_0 \in S_{\E_2}$ such that whenever $\mu_0 \in \pi_1^{-1}(T^*(\phi_2(y_0, \t_0)))$ is a Hahn-Banach extension of $T^*(\phi_2(y_0, \t_0))$ carried by $\ov{\Ch_{\H_1} K_1}$, then $\abs{\la \mu_0(\lbrace x_0 \rbrace), u \ra}>P$. We pick such a  $\mu_0$ and write it in the form $\mu_0=\psi_0 \ep_{x_0}+\nu_0$, where $\psi_0 \in \E_1$ and $\nu_0 \in \M(\ov{\Ch_{\H_1} K_1}, \E_1)$ with $\nu_0(\lbrace x_0 \rbrace)=0$. Then $\abs{\langle \psi_0, u \rangle}=\abs{\la \mu_0(\lbrace x_0 \rbrace), u \ra}>P$.

Now, choose $\ep \in (0, \frac{1}{2})$ such that $ \frac{1+3\ep}{1-\ep}<P$. Then, since \[\norm{\mu_0} \leq \min \lbrace \lambda(E_1), \lambda(E_2) \rbrace \leq 2,\] we have
\[\frac{1+\ep (\norm{\mu_0}+1)}{1-\ep}\leq \frac {1+3\ep}{1-\ep}<P,\] 
and we may choose a closed neighbourhood $V_1$ of $x_0$ such that $V_1 \subset V$ and 
\[\vert \langle \nu_0, u \rangle \vert(V_1)<P(1-\ep)-(1+\ep (\Vert \mu_0 \Vert+1)).\]
Since $x_0$ is a weak peak point, we may find a function $h_0 \in B_{\A_1}$, where $\A_1$ is the canonical scalar function space of $\H_1$, such that $h_0 \otimes u \in B_{\H_1}$, $h_0(x_0)>1-\ep$ and $\abs{h_0}< \ep$ on $\Ch_{\H_1} K_1 \setminus V_1$. Then, since $h_0$ is continuous, $\abs{h_0} \leq  \ep$ on the set 
\[ \ov{\Ch_{\H_1} K_1 \setminus V_1} \supseteq \ov{\Ch_{\H_1} K_1} \setminus V_1.\]
Then we have 
\begin{equation}
\nonumber
\begin{aligned}
&\abs{\langle \t_0, T(h_0 \otimes u)(y_0) \rangle}=\abs{\langle \t_0\ep_{y_0}, T(h_0 \otimes u) \rangle}=^{\eqref{projekce}}\abs{\langle \phi_2(y_0, \t_0), T(h_0 \otimes u)\rangle}
=\\&=
\abs{\langle T^*\phi_2(y_0, \t_0), h_0 \otimes u \rangle}=
\abs{\langle \mu_0, h_0 \otimes u \rangle}=\abs{\langle \psi_0\ep_{x_0} +\nu_0, h_0 \otimes u \rangle}
=\\&=
\abs{h_0(x_0)\langle \psi_0, u \rangle+ \int_{\ov{\Ch_{\H_1} K_1} \cap V_1} h_0~d\langle \nu_0, u \rangle+ \int_{\ov{\Ch_{\H_1} K_1} \setminus V_1} h_0~ d\langle \nu_0, u \rangle}
\geq \\& \geq
h_0(x_0)\abs{\la \psi_0, u \ra}-\int_{\ov{\Ch_{\H_1} K_1} \cap V_1} \abs{h_0} ~d\abs{\langle \nu_0, u \rangle}-\int_{\ov{\Ch_{\H_1} K_1} \setminus V_1} \abs{h_0}~ d\abs{\langle \nu_0, u \rangle}
> \\& >
(1-\ep)P-(P(1-\ep)-(1+\ep \Vert \mu_0 \Vert+\ep))-\ep\norm{\mu_0} = 1 + \ep.
\end{aligned}
\end{equation}
Thus $\norm{T(h_0 \otimes u)(y_0)}>1+\ep$. Since $y_\beta \rightarrow y_0$ and $T(h_0 \otimes u)$ is continuous, there exists a $\beta_0 \in B$ such that for all $\beta \geq \beta_0$ we have $\norm{T(h_0 \otimes u)(y_\beta)}>1+\ep$. Thus we can fix a $\beta \in B$ satisfying that $\norm{T(h_0 \otimes u)(y_\beta)}>1+\ep$ and $x_\beta=\rho_2(y_\beta) \notin V$. Then again by \eqref{prepis} and \eqref{duality} there exists $\t_\beta \in S_{\E_2}$ such that whenever $\mu_\beta \in \pi_1^{-1}(T^*(\phi_2(y_\beta, \t_\beta)))$ is a Hahn-Banach extension of $T^*(\phi_2(y_\beta, \t_\beta))$ carried by $\ov{\Ch_{\H_1} K_1}$, then $\abs{\langle \mu_\beta(\lbrace x_\beta \rbrace), u \rangle}> P$. We pick such a  $\mu_\beta$ and write it in the form $\mu_\beta=\psi_\beta \ep_{x_\beta}+\nu_\beta$, where $\psi_\beta \in \E_1$ and $\nu_\beta \in \M(\ov{\Ch_{\H_1} K_1}, \E_1)$ with $\nu_\beta(\lbrace x_\beta \rbrace)=0$. Then $\abs{\langle \psi_\beta, u \rangle}=\abs{\langle \mu_\beta(\lbrace x_\beta \rbrace), u \rangle}>P$. 
Next choose a closed neighbourhood $V_2$ of $x_\beta$ disjoint from $V$ such that
\[\vert \langle \nu_\beta, u \rangle \vert(V_2)<P(1-\ep)-(1+\ep (\Vert \mu_{\beta} \Vert+1)).\]
Since $x_\beta$ is a weak peak point, we may find a function $h_\beta \in B_{\A_1}$ such that $h_\beta \otimes u \in B_{\H_1}$, $h_\beta(x_\beta)>1-\ep$ and $\abs{h_\beta}< \ep$ on $\Ch_{\H_1} K_1 \setminus V_2$. Then as above we obtain that $\norm{T(h_\beta \otimes u)(y_\beta)}>1+\ep$. Now, by Lemma \ref{lambda} there exist $\alpha_1, \alpha_2 \in \ef$ such that $\abs{\alpha_1} \leq 1, \abs{\alpha_2} \leq 1$ and 
\begin{equation}
\nonumber
\begin{aligned}
&\norm{T(\alpha_1 (h_0 \otimes u)+\alpha_2(h_\beta \otimes u))}_{\sup} \geq \norm{\alpha_1 T(h_0 \otimes u)(y_\beta)+\alpha_2 T(h_\beta \otimes u)(y_\beta)}
>\\&>
(1+\ep)\lambda(E_2). 
\end{aligned}
\end{equation}
On the other hand, by the Phelps maximum principle (see \cite[Theorem 2.3.8]{fleming-jamison-1}) we have 
\begin{equation}
\nonumber
\begin{aligned}
&\norm{\alpha_1(h_0 \otimes u)+\alpha_2(h_\beta \otimes u)}_{\sup}\leq \norm{u}\norm{\alpha_1h_0+\alpha_2h_\beta}_{\sup}
=\\&=
\sup_{x \in K_1} \vert \alpha_1h_0(x)+\alpha_2 h_\beta(x) \vert=
\sup_{x \in \Ch_{\H_1} K_1} \vert \alpha_1h_0(x)+\alpha_2 h_\beta(x) \vert
 \leq 1+\ep.
\end{aligned}
\end{equation}
Thus, since $\Vert T \Vert \leq \min \lbrace \lambda(E_1), \lambda(E_2) \rbrace$ we obtain that 
\[\Vert T(\alpha_1(h_0 \otimes u)+\alpha_2(h_\beta \otimes u)) \Vert \leq \min \lbrace \lambda(E_1), \lambda(E_2) \rbrace (1+\ep) \leq \lambda(E_2)(1+\ep).\] This contradiction proves that $\rho_2$ is continuous. 

Analogously we would verify that $\rho_1$ is continuous.

This finishes the proof for the compact case. Now we assume that $K_1, K_2$ are locally compact and consider their one-point compactifications $J_i=K_i\cup\{\alpha_i\}$, where, for $i=1,2$, $\alpha_i$ is the point representing infinity. The spaces $\H_i$ are then closed subspaces of $\C(J_i,\ef)$ satisfying $h(\alpha_i)=0$, $h\in \H_i$. By Lemma~\ref{comp}, the assumption on weak peak points for $\Ch_{\H_i} J_i$ is satisfied. Thus the compact case implies the existence of a homeomorphism between $\Ch_{\H_1} J_1$ and $\Ch_{\H_2} J_2$. Since $\Ch_{\H_i} K_i$ is homeomorphic to $\Ch_{\H_i} J_i$ by Lemma \ref{comp}, the theorem follows.

\section{Cardinality of Choquet boundaries}
\label{card}

Before embarking on the proof of Theorem \ref{weak Banach-Stone} we prove several lemmas concerning the case when the Choquet boundaries are finite, inspired by \cite[Lemma 2.2]{candidogalego}.

\begin{lemma}
\label{l_inf}
Let $\H$ be a closed subspace of $\C(K, E)$ such that $\Ch_{\H} K$ is finite and each point of $\Ch_{\H} K$ is a weak peak point. Then $\H$ is isometrically isomorphic to $\ell^{\infty}(\Ch_{\H} K, E)$.
\end{lemma}

\begin{proof}
Let $\A$ be the canonical scalar function space of $\H$.
First we show that $\H$ is isometrically embedded into $\ell^\infty(\Ch_{\H} K, E)$ by the restriction mapping $r\colon \H\to \ell^\infty(\Ch_{\H} K, E)$. The mapping is an isometry by the maximum principle \cite[Theorem~2.3.8]{fleming-jamison-1}. Indeed, for each $h \in \H$ we have
\[\norm{h}_{\sup}=\sup_{\e \in S_{\E}, x \in K} \abs{\e(h)(x)}=\sup_{\e \in S_{\E}, x \in \Ch_{\H} K} \abs{\e(h)(x)}=\norm{r(h)}.\]
It remains to prove that the mapping $r$ is surjective.
Let $x\in \Ch_{\H} K$ and $e \in E$ be given. We show that there exists a function $h_{x, e}\in B_{\A}$ such that $h_{x, e}(x)=1$, $h_{x, e}=0$ on $\Ch_{\H} K\setminus \{x\}$ and $h_{x, e} \otimes e \in \H$. To this end we consider net $\{h_{U,\ep}\}$ in $B_{\A}$, where $U$ is a neighborhood of $x$, $\ep\in (0,1)$ and $h_{U,\ep}$ is a function satisfying $h_{U,\ep}(x)>1-\ep$, $\abs{h_{U,\ep}}<\ep$ on $\Ch_{\H} K\setminus U$ and $h_{U, \ep} \otimes e \in \H$. We consider the partial order on the set of pairs $(U,\ep)$ given by $(U_1,\ep_1)\le (U_2,\ep_2)$ provided $U_2\subset U_1$ and $\ep_2<\ep_1$. Since we know that $\H \subset \ell^\infty(\Ch_{\H} K, E)$, it follows that the space $\A \subset \ell^\infty(\Ch_{\H} K, \ef)$ is finite-dimensional. Thus $B_{\A}$ is compact in the norm topology, and the net $\{h_{U,\ep}\}$ possesses a cluster point $h_{x, e}\in B_{\A}$.  Then $h_{x, e}(x)=1$ and $h_{x, e}=0$ on $\Ch_{\H} K\setminus \{x\}$. Moreover, since each member of this net satisfies $h_{U,\ep}\otimes e \in \H$ and $\H$ is closed, we obtain that $h_{x, e} \otimes e \in \H$.

Now, any $f\in\ell^\infty(\Ch_{\H} K, E)$ can be written as
\[
f=\sum_{x\in \Ch_{\H} K}h_{x, f(x)}|_{\Ch_{\H} K} \otimes f(x),
\]
and thus $h=\sum_{x\in \Ch_{\H} K} h_{x, f(x)} \otimes f(x) \in \H$ satisfies $r(h)=f$.
\end{proof}

\begin{lemma}
\label{c0}
Let $\H$ be a closed subspace of $\C(K, E)$ such that $\Ch_{\H} K$ is infinite and each point of $\Ch_{\H} K$ is a weak peak point. Then for each $\ep>0$, $\H$ contains a $(1+\ep)$-isomorphic copy of $c_0$.
\end{lemma}

\begin{proof}
Let $\A$ be the canonical scalar function space of $\H$ and let $\ep>0$ be given. First observe that since $\Ch_{\H} K$ is infinite, there exists an open set $U \subset K$ such that $U$ contains a point of $\Ch_{\H} K$ and $\Ch_{\H} K \setminus \ov{U}$ is infinite. Indeed, choose arbitrary distinct points $x, y \in \Ch_{\H} K$ and let $U, V$ be open neighbourhoods of $x$ and $y$, respectively, such that $\ov{U} \cap \ov{V}=\emptyset$. Then if both $\Ch_{\H} K \setminus \ov{U}$ and $\Ch_{\H} K \setminus \ov{V}$ were finite, $\Ch_{\H} K$ would be finite as well. 

 Fix $e \in S_E$. We proceed by induction to find a sequence of functions $\{ f_n\}_{n=1}^\infty \subset B_{\A}$, points $\{ x_n\}_{n=1}^\infty \subset \Ch_{\H} K$ and pairwise disjoint open sets $\{U_n \}_{n=1}^\infty$, such that for each $n \in \en$, $h_n=f_n \otimes e \in \H$, $f_n(x_n)>1-\frac{\ep}{2^n}$, $\abs{f_n} <\frac{\ep}{2^n}$ on $\Ch_{\H} K \setminus U_n$ and $\Ch_{\H} K \setminus \bigcup_{i=1}^n \ov{U_i}$ is infinite.

 First we choose an open set $U_1 \subset K$ such that there exists a point $x_1 \in U_1 \cap \Ch_{\H} K$ and $\Ch_{\H} K \setminus \ov{U_1}$ is infinite. Since $x_1$ is a weak peak point, there exists a function $f_1 \in B_{\A}$ such that $f_1(x_1)>1-\frac{\ep}{2}$, $\abs{f_1}<\frac{\ep}{2}$ on $\Ch_{\H} K \setminus U_1$ and $f_1 \otimes e \in \H$. We denote $h_1=f_1 \otimes e$. Now suppose that $k \in \en$ and we have constructed finite sequences
$\{ h_n\}_{n=1}^k$, $\{ x_n\}_{n=1}^k$ and pairwise disjoint open sets $\{U_n \}_{n=1}^k$.
Then, since $\Ch_{\H} K \setminus \bigcup_{n=1}^k \ov{U_n}$  is infinite, there exists an open set $V_{k+1}$ intersecting $\Ch_{\H} K \setminus \bigcup_{n=1}^k \ov{U_n}$ and such that $(\Ch_{\H} K \setminus \bigcup_{n=1}^k \ov{U_n}) \setminus \ov{V_{k+1}}$ is infinite. Define $U_{k+1}=V_{k+1} \setminus \bigcup_{n=1}^k \ov{U_n}$. Next choose arbitrary $x_{k+1} \in U_{k+1} \cap \Ch_{\H} K$. Since $x_{k+1}$ is a weak peak point, there exists a function $f_{k+1} \in B_{\A}$ such that $f_{k+1}(x_{k+1})>1-\frac{\ep}{2^{k+1}}$, $\abs{f_{k+1}}<\frac{\ep}{2^{k+1}}$ on $\Ch_{\H} K \setminus U_{k+1}$ and $h_{k+1}=f_{k+1} \otimes e \in \H$. This finishes the construction.

Now, choose an arbitrary finite sequence $\{ \alpha_i \}_{i=1}^n$ of scalars. Choose $j \in \lbrace 1, \ldots, n \rbrace$ such that $\abs{\alpha_j}=\max \lbrace \abs{\alpha_i}, i=1, \ldots, n \rbrace$. Then we have

\begin{equation}
\nonumber
\begin{aligned}
&\norm{\sum_{i=1}^n \alpha_i h_i}_{\sup} \geq \abs{\alpha_j}\norm{h_j(x_j)}-\sum_{i \neq j} \abs{\alpha_i} \norm{h_i(x_j)} 
\geq \\& \geq \abs{\alpha_j}\norm{h_j(x_j)}-\abs{\alpha_j} \sum_{i \neq j} \norm{h_i(x_j)}  \geq \abs{\alpha_j} (1-\frac{\ep}{2}-\ep)\geq \abs{\alpha_j}(1-2\ep).
\end{aligned}
\end{equation}

Thus using \cite[Theorem~2.3.8]{fleming-jamison-1} we deduce that
\begin{equation}
\nonumber
\begin{aligned}
&(1-2\ep)\max_{i=1, \ldots, n} \abs{\alpha_i}\leq 
\norm{\sum_{i=1}^{n} \alpha_i h_i}_{\sup}
= \norm{e} \norm{\sum_{i=1}^{n} \alpha_i f_i}_{\sup} 
=\\&=
\sup_{x \in \Ch_{\H} K} \abs{\sum_{i=1}^n \alpha_i f_i(x)}\leq \max_{i=1, \ldots, n} \abs{\alpha_i} \sup_{x \in \Ch_{\H} K} \sum_{i=1}^n \abs{f_i(x)}\leq(1+\ep)\max_{i=1, \ldots, n} \abs{\alpha_i}.
\end{aligned}
\end{equation}
Thus it follows that if we define $T:c_0 \rightarrow \H$ by $T(\{ \alpha_n \}_{n=1}^\infty)=\sum_{n=1}^\infty \alpha_n h_n$, then $T$ is an isomorphism satisfying $\norm{T}\cdot\norm{T^{-1}}\leq \frac{1+\ep}{1-2\ep}$. Since $\ep>0$ is arbitrary, this finishes the proof.
\end{proof}

\begin{prop}
\label{finite}
Let for $i=1, 2$, $\H_i$ be a closed subspace of $\C(K_i, E_i)$ such that each point of $\Ch_{\H_i} K_i$ is a weak peak point. Let $E_1$ does not contain an isomorphic copy of $c_0$. Suppose that $\Ch_{\H_1} K_1$ is finite and $\H_1$ is isomorphic to $\H_2$. Then $\Ch_{\H_2} K_2$ is finite.
\end{prop}

\begin{proof}
By Lemma \ref{l_inf} there exists an $n \in \en$ such that $\H_1$, and thus also $\H_2$, is isomorphic to $\ell_n^{\infty}(E_1)$.
 Hence if $\Ch_{\H_2} K_2$ were infinite, then by Lemma \ref{c0}, $\ell_n^{\infty}(E_1)$ would contain an isomorphic copy of $c_0$. Thus by \cite[Theorem 1]{Samuel}, $E_1$ would contain an isomorphic copy of $c_0$ as well, contradicting our assumption. Thus $\Ch_{\H_2} K_2$ is finite.
\end{proof}

Now we embark on the proof of Theorem \ref{weak Banach-Stone}. We recall that a series $\sum_{i=1}^{\infty} e_i$ in a Banach space $E$ is \emph{weakly unconditionally Cauchy} if $\sum_{i=1}^{\infty} \abs{\la \e, e_i \ra}< \infty$ for each $\e \in \E$.

\emph{Proof of Theorem \ref{weak Banach-Stone}.} Using Lemma \ref{comp} we may assume that the spaces $K_1, K_2$ are compact. Moreover, we have to deal only with the case when both $\Ch_{\H_1} K_1$ and $\Ch_{\H_2} K_2$ are infinite by Proposition \ref{finite}.
Let $\pi_1\colon \M(K_1,\E_1)\to \H_1^{*}$ be the restriction mapping and let $\phi_2\colon K_2 \times B_{\E_2} \to B_{\H_2^*}$ be the evaluation mapping. For each $x\in \Ch_{\H_1} K_1$ we consider the function ${\chi}_{\{x\}}$. Let for $u \in S_{E_1}$, $\wh{\chi_{\{x\}}} \otimes u \colon \M(K_1, \E_1)\to \ef$ be defined as in Lemma~\ref{first-borel} and let $a_{x, u}^{**}\in \H_1^{**}$ satisfies $\la a_{x, u}^{**}, \pi_1(\mu) \ra=(\wh{\chi_{\{x\}}} \otimes u)(\mu)$ for $\mu$ carried by $\ov{\Ch_{\H_1} K_1}$ (see Lemma~\ref{faktor}). Then $a_{x, u}^{**}$ is of the first Borel class on $rB_{\H_1^*}$ for any $r>0$, see Lemma~\ref{first-borel} and~\ref{pi}(b).

We show that $\abs{\Ch_{\H_1} K_1} \leq \abs{\Ch_{\H_2} K_2}$. For $y \in \Ch_{\H_2} K_2$ we denote
\[X_y=\lbrace x \in \Ch_{\H_1} K_1: \sup_{u \in S_{E_1}, \t \in S_{\E_2}} \abs{\langle T^{\ast\ast}(a_{x, u}^{\ast\ast}), \phi_2(y, \t)\rangle} >0 \rbrace .\]	
First we show that for each $x \in \Ch_{\H_1} K_1$ there exists $y \in \Ch_{\H_2} K_2$ such that $x \in X_y$. Assuming the contrary, there exists $x \in \Ch_{\H_1} K_1$ such that for all $y \in \Ch_{\H_2} K_2, u \in S_{E_1}$ and $\t \in S_{\E_2}$:
\[ \abs{\langle T^{\ast\ast}(a_{x, u}^{\ast\ast}), \phi_2(y, \t)\rangle} =0.\]	
Then for arbitrary $u \in S_{E_1}$ using Lemma \ref{principle}(i) we have 
\[0=\sup_{y \in \Ch_{\H_2} K_2, \t \in S_{\E_2}} \abs{ \langle T^{\ast\ast}(a_{x, u}^{\ast\ast}), \phi_2(y, \t)\rangle}=\norm{T^{\ast\ast}(a_{x, u}^{**})} > 0,\]
which is a contradiction.

Now we show that for each $y \in \Ch_{\H_2} K_2$, the set $X_y$ is at most countable. Suppose not. Then we fix $y \in \Ch_{\H_2} K_2$ with $X_y$ uncountable.
For $\t \in S_{\E_2}$ we denote by $\M_{\t}$ the set of vector measures $\mu \in \pi_1^{-1}(T^\ast(\phi_2(y, \t))) \cap \M(K_1, \E_1)$ satisfying $\norm{\mu}=\norm{T^\ast(\phi_2(y, \t))}$, and carried by $\ov{\Ch_{\H_1} K_1}$. Each such measure satisfies $\norm{\mu} \leq \norm{T^*}=\norm{T}$.
Since the set $X_y$ is uncountable, there exists an $\ep>0$ such that the set 
\[X_y^{\ep}=\{x \in \Ch_{\H_1} K_1 \exists \t \in S_{\E_2} \exists u \in S_{E_1}:\abs{\langle T^{\ast\ast}(a_{x, u}^{\ast\ast}), \phi_2(y, \t)\rangle} >\ep \}\]
is infinite.

By \eqref{duality} we know that for each $\mu \in \M_{\t}$ it holds that 
\[\la \mu(\{x\}), u \ra=\langle T^{\ast\ast}(a_{x, u}^{\ast\ast}), \phi_2(y, \t)\rangle,\] 
and thus we see that the set $X_y^{\ep}$ coincides with
\begin{equation}
\nonumber
\lbrace x \in \Ch_{\H_1} K_1 \exists \t \in S_{\E_2} \forall \mu \in \M_{\t}: \norm{\mu(\lbrace x \rbrace)} > \ep\rbrace.
\end{equation}

By the proof of Lemma \ref{c0}, there exists a sequence $\{ x_n\}_{n=1}^{\infty} \subset X_y^{\ep}$ and a sequence $\{U_n\}_{n=1}^{\infty}$ of pairwise disjoint open subsets of $K_1$ such that $x_n \in U_n$ for each $n \in \en$. For each $n \in \en$, since $x_n \in X_y^{\ep}$, there exists $\t_n \in S_{\E_2}$, $\mu_n \in \M_{\t_n}$ and $u_n \in S_{E_1}$ such that $\abs{\la \mu_n(\{x_n\}), u_n \ra }> \ep$. Now we find $n_0 \in \en$ such that 
\[ \frac{1}{2^{n_0}}\norm{T}<\ep(\frac{1}{2}-\frac{1}{2^{n_0}}).\]

By making the sets $U_n$ smaller if necessary, we may further assume that
\begin{equation}
\label{reg}
\abs{\la \mu_n, u_n \ra}(U_n \setminus \{ x_n \}) < \ep(\frac{1}{2}-\frac{1}{2^{n_0}})-\frac{1}{2^{n_0}}\norm{T}
\end{equation}
 for each $n \in \en$.
Moreover, for each $n \in \en$ we can find a function $f_n \in B_{\A_1}$, where $\A_1$ is the canonical scalar function space of $\H_1$, such that $h_n=f_n \otimes u_n \in \H_1$,
 \[ f_n(x_n)>
1-\frac{1}{2^{n_0+n}} \quad \text{and} \quad \abs{f_n} < \frac{1}{2^{n_0+n}} \quad \text{on} \quad \Ch_{\H_1} K_1 \setminus U_n. \] 

Now we claim that the series $ \sum_{n=1}^{\infty}Th_n (y)$ in $E_2$ is weakly unconditionally Cauchy. To this end, first observe that for each $n \in \en$ and scalars $\alpha_1, \ldots \alpha_n \in S_{\ef}$  we have 
by \cite[Theorem~2.3.8]{fleming-jamison-1} that 
\[ \norm{\sum_{i=1}^n \alpha_i h_i}_{\sup} \leq \max_{i=1, \ldots, n} \norm{u_i} \sup_{x \in \Ch_{\H_1} K_1} \sum_{i=1}^n \abs{\alpha_i f_i(x)} <2.\]

For a given $v^* \in S_{\E_2}$, we choose an arbitrary $\mu \in \M_{v^*}$, and for each $n \in \en$ we find $\alpha_1, \ldots \alpha_n \in S_{\ef}$ such that $\abs{\la \mu, h_i \ra}=\alpha_i \la \mu, h_i \ra$ for each $i=1, \ldots, n$. Then we have
\begin{equation}
\nonumber
\begin{aligned}
&\sum_{i=1}^n \abs{ \la v^*, Th_i(y) \ra}=\sum_{i=1}^n \abs{ \la \phi_2(y, v^*), Th_i \ra}=
\sum_{i=1}^n \abs{ \la T^*\phi_2(y, v^*), h_i \ra}
=\\&=
\sum_{i=1}^n \abs{ \la \mu, h_i \ra}=\sum_{i=1}^n \alpha_i \la \mu, h_i \ra=\la \mu, \sum_{i=1}^n \alpha_i h_i \ra  \leq
 \norm{\mu} \norm{\sum_{i=1}^n \alpha_i h_i} \leq 2 \norm{\mu}.
\end{aligned}
\end{equation}

Thus also $\sum_{i=1}^{\infty} \abs{ \la v^*, Th_i(y) \ra} \leq 2\norm{\mu} < \infty$, and the series is weakly unconditionally Cauchy.

Now we show that the norms of the members of the series are uniformly bounded away from zero. For each $i \in \en$ we have 

 \begin{equation}
\nonumber
\begin{aligned}
&\norm{Th_i(y)} \geq \abs{\la \t_i, T(f_i \otimes u_i)(y) \ra}=\abs{\la \phi_2(y, \t_i), T(f_i \otimes u_i) \ra}
=\\&=
\abs{\la T^*(\phi_2(y, \t_i)), f_i \otimes u_i \ra}=\abs{\la \mu_i, f_i \otimes u_i \ra}=\abs{\int_{K_1} f_i~d\la \mu_i, u_i \ra}
\geq \\& \geq 
\int_{\lbrace x_i \rbrace} f_i~d\la \mu_i, u_i \ra-\abs{\int_{U_i \setminus \lbrace x_i \rbrace} f_i~d\la \mu_i, u_i \ra}-\abs{\int_{\ov{\Ch_{\H_1} K_1} \setminus U_i} f_i~d\la \mu_i, u_i \ra}
> \\& >
(1-\frac{1}{2^{n_0+i}})\ep-\abs{\la \mu_i, u_i \ra} (U_i \setminus \{ x_i \})-\frac{1}{2^{n_0+i}}\norm{\mu_i} \geq \\& \geq 
(1-\frac{1}{2^{n_0}})\ep-\abs{\la \mu_i, u_i \ra} (U_i \setminus \{ x_i \})-\frac{1}{2^{n_0}}\norm{T}
\geq^{\eqref{reg}} \frac{\ep}{2}.
\end{aligned}
\end{equation}
Thus we obtained that the series $ \sum_{n=1}^{\infty}T(h_n) (y)$ in $E_2$ is weakly unconditionally Cauchy, but $\inf_{n \in \en} \norm{T(h_n) (y)}>0$. This by \cite[Theorem 6.7]{morrison2001functional} means that $E_2$ contains an isomorphic copy of $c_0$, contradicting our assumption. This contradiction shows that $X_y$ is at most countable for all $y \in \Ch_{\H_2} K_2$. 

Thus, since we know that
\[\Ch_{\H_1} K_1 =\bigcup_{y \in \Ch_{\H_2} K_2} X_y,\]
we conclude that $\abs{\Ch_{\H_1} K_1} \leq \abs{\Ch_{\H_2} K_2}$. By reversing the role of $\Ch_{\H_1} K_1$ and $\Ch_{\H_2} K_2$ we obtain the reverse inequality, which concludes the proof.


\bibliography{iso-functions}\bibliographystyle{siam}
\end{document}